\documentclass[reqno]{amsart}
\usepackage{amssymb,enumerate,mathrsfs,amscd}

\numberwithin{equation}{section}

\newtheorem{theorem}[equation]{Theorem}
\newtheorem{proposition}[equation]{Proposition}
\newtheorem{lemma}[equation]{Lemma}
\newtheorem{corollary}[equation]{Corollary}

\theoremstyle{definition}
\newtheorem{definition}[equation]{Definition}
\newtheorem{remark}[equation]{Remark}
\newtheorem{example}[equation]{Example}

\def\C{\mathbb C}
\def\N{\mathbb N}
\def\R{\mathbb R}

\def\L{\mathscr L}
\def\M{\mathcal M}

\def\S{\mathscr S}
\def\Y{\mathcal Y}
\def\Z{\mathcal Z}

\def\Sing{\mathcal E}
\def\AllSing{\mathcal S}

\def\bP{\,{}^b\!P}
\def\trb{{\mathscr T}}
\def\dbar{d\hspace*{-0.08em}\bar{\hspace*{0.05em}}}
\def\eps{\varepsilon}
\def\im{i}
\def\m{\mathfrak m}

\def\set#1{\{#1\}}
\def\bPhat{\,{}^b\!\widehat P}

\def\display#1#2{\mbox{\parbox{#1} {#2}}}
\def\bysame{\underline{\hspace{25pt}}}

\DeclareMathOperator{\Aut}{Aut}
\DeclareMathOperator{\Cl}{Cl}

\DeclareMathOperator{\End}{End}
\DeclareMathOperator{\diam}{diam}
\DeclareMathOperator{\ind}{ind}
\DeclareMathOperator{\tind}{t-ind}
\DeclareMathOperator{\spec}{spec}
\DeclareMathOperator{\Hom}{Hom}
\DeclareMathOperator{\Op}{Op}
\DeclareMathOperator{\sym}{ \sigma\!\!\!\sigma}

\begin{document}

\title{Elliptic systems of variable order}

\thanks{Work partially supported by the National Science Foundation, Grants DMS-0901202 and DMS-0901173}

\author{Thomas Krainer}
\address{Penn State Altoona\\ 3000 Ivyside Park \\ Altoona, PA 16601-3760}
\email{krainer@psu.edu}
\author{Gerardo A. Mendoza}
\address{Department of Mathematics\\ Temple University\\ Philadelphia, PA 19122}
\email{gmendoza@temple.edu}

\begin{abstract}
The general theory of boundary value problems for linear elliptic wedge operators (on smooth manifolds with boundary) leads naturally, even in the scalar case, to the need to consider vector bundles over the boundary together with general smooth fiberwise multiplicative group actions. These actions, essentially trivial (and therefore invisible) in the case of regular boundary value problems, are intimately connected with what passes for Poisson and trace operators, and to pseudodifferential boundary conditions in the more general situation. Here the part of the theory pertaining pseudodifferential operators is presented in its entirety. The symbols for the latter operators are defined with the aid of an intertwining of the actions. Also presented here are the ancillary Sobolev spaces, an index theorem for the elliptic elements of the pseudodifferential calculus, and essential ingredients for analyzing boundary conditions of Atiyah-Patodi-Singer type in the more general theory.
\end{abstract}


\maketitle


\section{Introduction}\label{sec-Introduction}

In this note we introduce a calculus of pseudodifferential operators of variable order that act on sections of vector bundles endowed with smooth multiplicative group action over a closed manifold. The need to develop such a calculus arose from ongoing work by the authors in \cite{KrMe13a} aimed at developing a theory of boundary value problems for elliptic wedge operators. In the next few paragraphs we briefly describe this problem in order to give some motivation to the present work.

Elliptic wedge operators are structurally modeled on the operators one obtains by rewriting a regular linear differential operator in cylindrical coordinates along a submanifold. Thus the general form for such an operator is 
\begin{equation}\label{WedgeOperator}
A=x^{-m}\sum_{k+|\alpha|+|\beta| \leq m} a_{k,\alpha,\beta}(x,y,z) (x D_x)^k(x D_y)^\alpha D_z^\beta
\end{equation}
as one sees after some manipulation; $x$ is the radial variable, valued in $[0,\eps)$ for some $\eps>0$, $y$ the axial variable, ranging over an open set in a manifold $\Y$ of dimension $q$ and called the edge, and $z$ the variable in a general compact manifold $\Z$, a sphere in the case of cylindrical coordinates. The coefficients $a_{k\alpha\beta}$ are smooth up to $x=0$ (see Schulze \cite{SchuNH}). Operators of the form $P=x^mA$ are called edge operators (see Mazzeo \cite{Maz91}). In the general set-up for edge operators, the boundary (here given by $x=0$) of the manifold is the total space of a fiber bundle over $\Y$ with compact fibers $\Z_y$. Ellipticity, assumed throughout this introduction, means that
\begin{equation*}
\sum_{k+|\alpha|+|\beta| = m} a_{k,\alpha,\beta}(x,y,z) \xi^k\eta^\alpha \zeta^\beta
\end{equation*}
is invertible when $(\xi,\eta,\zeta)\ne 0$.

Let $\AllSing_{y,\sigma}$ be the set of finite sums
\begin{equation*}
\tau=\sum_{\ell} \phi_{\sigma,\ell}\, x^{\im\sigma}\log^\ell x;
\end{equation*}
here $y\in \Y$ and $\sigma\in\C$ are arbitrary and $\phi_{\sigma,\ell}$ is a section along $\Z_y$ of the vector bundle on which $A$ acts. The first link to $A$ is the subspace $\Sing_{y,\sigma}\subset\AllSing_{y,\sigma}$ consisting of those elements solving the equation
\begin{equation*}
\bP_y\tau=0, \quad \bP_y=\sum_{k+|\beta| \leq m} a_{k,0,\beta}(0,y,z) (xD_x)^k D_z^\beta.
\end{equation*}
This is a finite-dimensional space (and its elements have smooth coefficients $\phi_{\sigma,\ell}$) because of ellipticity. The set of elements $\sigma$ for which $\Sing_{y,\sigma}\ne 0$ is the boundary spectrum of $\bP$ (or $A$) at $y$ (see Mazzeo, op. cit., Melrose \cite{Mel93}, Krainer and Mendoza \cite{KrMe12a}), denoted $\spec_{b}(\bP_y)$. They are the complex numbers for which the indicial family at $y$,
\begin{equation*}
\bPhat_y(\sigma)=\sum_{k+|\beta| \leq m} a_{k,0,\beta}(0,y,z) \sigma^k D_z^\beta,
\end{equation*}
has nontrivial kernel. Fix some $\gamma\in \R$ and assume 
\begin{equation*}
\spec_{b}(\bP_y)\cap\set{\sigma\in \C:\Im\sigma=\gamma,\gamma-m}=\emptyset.
\end{equation*}
Then, as shown in \cite{KrMe12a}, 
\begin{equation}\label{FiberOfTrace}
\trb_y=\bigoplus_{\gamma-m<\Im\sigma<\gamma}\Sing_{y,\sigma}
\end{equation}
is the fiber over $y$ of a smooth vector bundle $\trb\to \Y$, the trace bundle of $A$ (the number $\gamma$ is implicit). We should perhaps point out that the notion of smoothness of $\trb$ is not trivial because of the possible branching behavior of $\spec_{b}(\bP_y)$.

In the case of a classical elliptic differential operator of order $m$, the boundary spectrum is $\set {-\im k: k=0\dotsc,m-1}$ and the spaces $\Sing_{y,-\im k}$ reduce to $\phi_{\sigma,0}x^k$ with $\sigma=-\im k$. To see this, suppose 
\begin{equation*}
A=\sum_{k+|\alpha|\leq m} a_{k,\alpha}(x,y)D_x^k D_y^\alpha
\end{equation*}
is such an operator; the fibers $\Z_y$ are just the points of $\Y$ and the coefficients are smooth. As above, the boundary is $x=0$ while the interior of the manifold lies in $x>0$. The operator $P=x^mA$ is given by
\begin{equation*}
P=\sum_{k+|\alpha|\leq m} a_{k,\alpha}(x,y)x^{m-k-|\alpha|}p_k(xD_x+\im|\alpha|) (xD_y)^\alpha
\end{equation*}
with $p_k(\sigma)=(\sigma+\im(k-1))(\sigma+\im(k-2))\cdots\sigma$, and so 
\begin{equation*}
\bP_y=a_{m,0}(0,y)p_m(xD_x).
\end{equation*}
The solutions of $\bP_y\tau=0$ are just polynomials of degree $\leq m-1$,
\begin{equation}\label{RegularTraceFiber} 
\tau=\sum_{j=0}^{m-1} \phi_{j}(y)x^j.
\end{equation}
These spaces of polynomials are, in this case, the fibers \eqref{FiberOfTrace} of the trace bundle of the regular elliptic operator $A$. The powers $j$ correspond to the numbers of the form $\im\sigma$ with $\sigma$ a root of the indicial family of $A$ as claimed: the indicial family is $a_{m,0}(0,y)p_m(\sigma)$. Observe that the ellipticity of $A$ ensures that $a_{m,0}(0,y)$ is invertible. We shall not enter further here into details about this (the reader may consult \cite{KrMe12a,KrMe13a} for more information) except to point out that the polynomials \eqref{RegularTraceFiber} are the terms forming the Taylor polynomials in $x$ of degree $m-1$ of putative solutions of $Au=f$ at the boundary (in this case $\gamma=-1/2$), and that classical boundary conditions are placed on the coefficients $\phi_j$ as functions (or sections if $A$ is not a scalar operator) over $\Y$.

The nature of $\trb_y$ (and of the trace bundle) is rather more intricate in the general case. Sections of $\trb$ do play the same role in the general theory as polynomials in the standard theory, therefore boundary conditions are pseudodifferential conditions on sections of this trace bundle. One may attempt at first to take standard operators at this stage. Note, however, that even if the boundary spectrum is simple (but non-constant), the pertinent distributional sections of $\trb$ will naturally have varying regularity in $y$ depending on the factor $x^{\im\sigma}$. This explains why the pseudodifferential operators need to be adapted. A second place where the theory is needed is in the construction of Poisson and trace operators. However we will not discuss this here but refer to our forthcoming work \cite{KrMe13a}.

\medskip
Having introduced the motivating vector bundles, we now address the group action. The operator $x\partial_x$ acts on $\AllSing_{y,\sigma}$ and since it commutes with $\bP_y$, it preserves the spaces $\Sing_{y,\sigma}$, hence acts on the fibers of $\trb$. The space $\Sing_{y,\sigma}$ is the generalized eigenspace of $x\partial_x$ in the fiber $\trb_y$ associated with the eigenvalue $i\sigma$.
That $x\partial_x$ acts smoothly as an endomorphism $\trb\to\trb$ requires an understanding of the meaning of the $C^\infty$ structure of $\trb$ which we again omit (but direct the reader to \cite{KrMe12a}). In the classical case, where the fiber $\trb_y$ consists of the polynomials \eqref{RegularTraceFiber}, the $C^\infty$ sections are those whose coefficients $\phi_j$ are smooth functions of $y$. In this case $x\partial_x$ acts on $\trb_y$ quite trivially: the eigenvalues are the numbers $j$, the corresponding eigenspace consist of monomials of degree $j$, and evidently $x\partial_x$ gives a smooth endomorphism of $\trb$. In the general case $x\partial_x$ acts, as already asserted, on each $\Sing_{\sigma,y}$, but now these are generalized eigenspaces (the $\log$ terms may be present) corresponding to the eigenvalue $\im \sigma$. As $y$ varies, so may the eigenvalues. Even so, the operator $x\partial_x$ is a smooth  endomorphism of $\trb$. The $\R_+$-action generated by $x\partial_x, \kappa_{\varrho} = \varrho^{x\partial_x} \in C^{\infty}(\Y;\End(\trb))$, is simply the one that is fiberwise based on the formula $(\kappa_\varrho f)(x) =f (\varrho x)$.

\medskip
Abstracting, we shall consider vector bundles $E$, $E_1$,\dots\ over a smooth manifold $\Y$ of dimension $q$ together with endomorphisms $a\in C^\infty(\Y;\End(E))$, $a_1\in C^\infty(\Y;\End(E_1))$,\dots, and build up a theory of pseudodifferential operators based on symbol classes that intertwine the $\R_+$-actions generated by these endomorphisms. We do this by first observing, in Section~\ref{sec-deltaAdmiss}, that over sufficiently small open sets $\Omega$ the eigenvalues of the various infinitesimal generators cluster in sets of small diameter $\delta<1$. This brings with it a decomposition of the part over $\Omega$ of the vector bundle into a direct sum of subbundles on each of which the generator is almost constant from fiber to fiber, giving us enough control on sizes of derivatives of the action to allow us to define, in Section~\ref{sec-localsymbols}, symbols of H\"ormander type $(1,\delta)$ (see \cite{Hormander1971}) that are twisted by the actions. When the generators are constant block-diagonal, the symbols become of Douglis-Nirenberg type, see \cite{ChazarainPiriou}; the discussion in this reference starting on page 295 and dealing with boundary value problems is particularly illuminating.

In addition to the local definition of the symbols, Section~\ref{sec-localsymbols} contains the basic elements necessary to form a viable local theory of pseudodifferential operators. The most fundamental result in connection with this is Proposition~\ref{basicprops}, one of whose assertions relates our symbol classes with the standard H\"ormander classes of type $(1,\delta)$; this gives a considerable simplification of the proofs in Section \ref{sec-localoperators} of composition, invariance under changes of coordinates, and existence of adjoints and asymptotic summability in the class. The basis for the eventual globalization is Corollary~\ref{ChangeofFrame}. Incidentally, the number $\delta$, other than lying in the interval $(0,1)$, is completely arbitrary and can be taken as small as one wishes in a particular application. In this paper it is fixed once and for all.

The local definition of the pseudodifferential operators is given in Section~\ref{sec-localoperators}. The approach here is quite classical in that we take advantage of the relation with H\"ormander classes just mentioned. The proof of composition formulas, for instance, requires almost no extra work. Under the natural notion of ellipticity we prove existence of parametrices in the calculus. The section ends with Proposition~\ref{InvarianceChangeOfFrame} on changes of frame. This is necessary to account for coverings of the original manifold by open sets for which the eigenvalues of the infinitesimal generators cluster in different ways over overlaps, and together with invariance under diffeomorphisms this allows for globalization in Section \ref{sec-global calculus}.

The local versions of Sobolev spaces adapted to the action (elements in a given space have, in addition to a constant shift in regularity, variable smoothness as determined by the infinitesimal generator) are constructed in Section~\ref{sec-localspaces}. We also prove there mapping properties, including regularity results for the elliptic elements in our calculus.

Section~\ref{sec-global calculus} deals with the global definition and some properties of operators from the global perspective. We prove, in particular, the existence of an exactly invertible operator that changes order (in the same vein as $(1-\Delta)^{s/2}$ for regular Sobolev spaces in $\R^n$). This is a useful tool, in particular in the following section on the global Sobolev spaces of variable smoothness. 

We define global versions of Sobolev spaces in Section \ref{sec-Fredholm theory}. Having the spaces at hand we also prove here Fredholm properties and existence of parametrices, and establish an Atiyah-Singer index theorem for elliptic elements. 

Finally, in Section~\ref{sec-Toeplitz Operators} we prove a theorem tailored for analyzing boundary conditions of generalized Atiyah-Patodi-Singer type \cite{APS} in the general theory.

We end this introduction with some remarks. First, pseudodifferential operators of variable order and associated Sobolev spaces in the scalar case are classical, see for example \cite{Beals1975, Leopold1991, Unterberger1, Unterberger2}. However, vector-valued analogues of these spaces are not suitable to capture the behavior of traces along the edge for functions in domains of natural $L^2$-based extensions of elliptic wedge operators. Second, our calculus contains naturally, as a special case, the theory of Douglis-Nirenberg elliptic systems, and our index theorem accordingly specializes to an index theorem for such systems. Third, our local theory recovers in the special case that the generators are constant (independent of the base variable $y$) the calculus of pseudodifferential operators with twisted operator valued symbols and ${\mathcal W}$-Sobolev spaces introduced by Schulze (when specialized to the finite-dimensional situation), see \cite{SchuNH}.


\section{$\delta$-admissibility}\label{sec-deltaAdmiss}

Let $\Y$ be a smooth manifold and $E\to \Y$ be a smooth complex vector bundle of rank $M$ and $a\in C^\infty(\Y;\End(E))$. Fix $\delta\in (0,1)$ arbitrarily. 

\begin{definition} Let $\Omega\subset \Y$ be open. A $\delta$-admissible decomposition of $E$ (relative to $a$) over $\Omega$ is a decomposition of $E_\Omega$, the part of $E$ over $\Omega$, as a direct sum of $a$-invariant trivial subbundles $E_k\to \Omega$ for which the closures, $\Sigma_k$, of the sets
\begin{equation*}
\bigcup_{y\in\Omega}\spec(a(y)|_{E_k})
\end{equation*}
are pairwise disjoint and of diameter $<\delta$. Here $\spec(a(y)|_{E_k})$ denotes the spectrum of the part $a(y)|_{E_k} : E_k \to E_k$ of $a(y)$ in $E_k$. The sets $\Sigma_k$ are referred to as eigenvalue clusters.
\end{definition}

Every point of $\Y$ lies in an open set $\Omega$ over which there is a $\delta$-admissible decomposition of $E$. Namely, let $\set{\sigma_k}_{k=1}^N$ be an enumeration of the points of $\spec a(y_0)$, pick numbers $0<\delta_k<\delta$ such that the disks $D(\sigma_k,\delta_k)=\set{\sigma:|\sigma-\sigma_k|\leq\delta_k}$ are pairwise disjoint, and let $\Omega$ be a neighborhood of $y_0$ such that 
\begin{equation*}
\spec (a(y))\subset \bigcup_{k=1}^N D(\sigma_k,\delta_k/2)\text{ for all }y\in \Omega.
\end{equation*}
Now let 
\begin{equation*}
\Pi_{k,y}=\frac{1}{2\pi\im}\int_{|\sigma-\sigma_k|=\delta_k}(\sigma-a(y))^{-1}\,d\sigma,\ y\in \Omega.
\end{equation*}
Then the spaces
\begin{equation*}
E_{k,y}=\Pi_{k,y}E_y,
\end{equation*}
which are $a(y)$-invariant, join to give smooth vector subbundles $E_k$ of $E_\Omega$ which are trivial if $\Omega$ is a small enough.

\begin{definition}
Let $\Omega\subset \Y$ be open. A $\delta$-admissible trivialization of $E$ over $\Omega$ (relative to $a$) is a trivialization of the part of $E$ over $\Omega$ that respects a $\delta$-admissible decomposition of $E_\Omega$.
\end{definition}

In other words, the trivialization of $E_\Omega$ is of the form $\phi=\bigoplus\phi_k$ where $\phi_k$ is a trivialization of $E_k$. For such a trivialization $\phi$, let $a_\phi=\phi a\phi^{-1}$, which we may, and do, view simply as a smooth map $\Omega\to \End(\C^M)$. The following properties of $a_\phi$, listed for convenience of reference, are a reflection of the $\delta$-admissibility of $\phi$:
\begin{equation}\label{Edecomp}
\display{310pt}{
\begin{enumerate}
\item There is a decomposition $\C^M=\bigoplus_{k=1}^N V_k$ by $a_\phi$-invariant subspaces.
\item The eigenvalue cluster sets
\begin{equation*}
\Sigma_k=\Cl\big(\bigcup_{y\in \Omega}\spec(a_\phi(y)|_{V_k})\big)
\end{equation*}
are pairwise disjoint with $\diam(\Sigma_k)<\delta$.
\end{enumerate}
}
\end{equation}

The element $a\in C^\infty(\Y;\End(E))$ generates a multiplicative group 
\begin{equation*}
\R_+\ni \varrho\mapsto \varrho^a\in C^\infty(\Y;\Aut(E)),
\end{equation*}
fiberwise expressed as 
\begin{equation}\label{DunfordIntegralGroup}
\varrho^{a(y)} = \frac{1}{2\pi i}\int_{\Gamma} \varrho^{\sigma}(\sigma - a(y))^{-1}\,d\sigma
\end{equation}
for all $\varrho > 0$, where $\Gamma$ is any fixed contour of integration that encloses $\spec(a(y))$. Recall that the multiplicative group property means that $(\varrho_1\varrho_2)^{a(y)} = \varrho_1^{a(y)}\varrho_2^{a(y)}$ for $\varrho_1,\varrho_2 > 0$, and $1^{a(y)} = \textup{Id}$.

If $\phi$ is a $\delta$-admissible trivialization of $E$ over $\Omega$ then the formula holds for every $y\in \Omega$ and fixed suitable $\Gamma$. In particular, this shows that the function 
\begin{equation*}
\Omega\times\R_+\ni(y,\varrho)\mapsto \varrho^{a_\phi(y)}\in \Aut(\C^M)
\end{equation*}
is smooth. In this local context it will be important to keep in mind that one can choose the contour $\Gamma$ to be of the form
\begin{equation*}
\Gamma = \Gamma_{1} \cup \dots \cup \Gamma_{N}, 
\end{equation*}
where for each $k$, 
\begin{equation}\label{SmallGammak}
\display{300pt}{$\Gamma_{k}$ encloses the compact set $\Sigma_{k}$, has winding number $0$ with respect to each point $\sigma'\in\Sigma_{k'}$, $k'\ne k$, and has diameter $<\delta$.}
\end{equation}
The group $\varrho^{a_\phi(y)}$ is block-diagonal with respect to the decomposition \eqref{Edecomp}, with the block in $V_{k}$ being the group generated by $a_\phi(y)|_{V_k}$ in $V_{k}$.

If $\phi$ is a $\delta$-admissible trivialization of $E$ over some open set $\Omega\subset \Y$ and $a_\phi=\phi a\phi^{-1}$ then of course
\begin{equation*}
\varrho^a = \phi^{-1}\varrho^{a_\phi}\phi 
\end{equation*}
over $\Omega$, and if $\psi$ is another $\delta$-admissible trivialization over an open set $\Omega'$ and $a_\psi=\psi a \psi^{-1}$, then the above formula coupled with the analogous formula for $\psi$ gives
\begin{equation*}
\varrho^{a_\psi}=(\psi\phi^{-1})\varrho^{a_\phi}(\psi\phi^{-1})^{-1}
\end{equation*}
on $\Omega\cap \Omega'$, equivalently,
\begin{equation}\label{ZeroHomogeneus}
(\psi\phi^{-1})=\varrho^{-a_\psi} (\psi\phi^{-1})\varrho^{a_\phi}.
\end{equation}
This formula should be viewed as expressing a property of the transition functions associated to $\delta$-admissible trivializations in terms of the multiplicative actions generated by $a_\phi$ and $a_\psi$. It is a fundamental component in the globalization of our theory whose analytic consequence is stated in Corollary~\ref{ChangeofFrame}.

\medskip
In the following three sections, which deal with the local theory, we confine ourselves to actions on various complex Euclidean spaces coming from $\delta$-admissible trivializations of various vector bundles $E$ (or $E_1$ and $E_2$) and their respective (given) infinitesimal generators of multiplicative actions. Omitting a reference to the particular $\delta$-admissible trivializations, infinitesimal generators of the actions are still denoted $a$ (or $a_1$ and $a_2$ as the case may be) and are simply smooth maps $\Omega\to \End(\C^M)$ (or $\C^{M_1}$ and $\C^{M_2}$). In all cases the underlying assumption is that there is a $\delta$-admissible decomposition of the respective Euclidean space as described in \eqref{Edecomp}.


\section{The symbols in the local calculus}\label{sec-localsymbols}

Let $\Omega\subset\R^q$ be open, $a_j\in C^\infty(\Omega,\End(\C^{M_j}))$, $j=1,2$. Fix $\delta\in (0,1)$. We assume throughout this and the next two sections that \eqref{Edecomp} holds in $\Omega$ both for $a_1$ and $a_2$, eventually also for any of the infinitesimal generators $a\in C^\infty(\Omega,\End(\C^M))$ of the group actions we discuss. Of course the decomposition in part (1) of \eqref{Edecomp} and what the eigenvalue cluster sets in part (2) are may depend on $a$ or the $a_j$.

\begin{definition}\label{symboldef}
Let $\mu \in \R$. We define
\begin{equation*}
S_{1,\delta}^{\mu}(\Omega\times\R^q;(\C^{M_1},a_1),(\C^{M_2},a_2))
\end{equation*}
to be the space of all $p(y,\eta) \in C^{\infty}(\Omega\times\R^q,\Hom(\C^{M_1},\C^{M_2}))$ such that for every compact subset $K \Subset \Omega$ and all $\alpha,\beta \in \N_0^q$ there exists a constant $C_{K,\alpha,\beta} > 0$ such that
$$
\|\langle \eta \rangle^{a_2(y)}\bigl(D_y^{\alpha}\partial_{\eta}^{\beta}p(y,\eta)\bigr)\langle \eta \rangle^{-a_1(y)}\| \leq C_{K,\alpha,\beta} \langle \eta \rangle^{\mu - |\beta| + \delta|\alpha|}
$$
for all $(y,\eta) \in K\times\R^q$. Here and elsewhere $\langle \eta \rangle=\sqrt{1+|\eta|^2}$.
\end{definition}

If the $a_j$ are constant, we may allow $\delta$ to be $0$. Further, if $a_j \equiv 0$ for all $y \in \Omega$ then $\varrho^{a_j} \equiv \textup{Id}$ is the trivial action on $\C^{M_j}$, and in this case we will just write $\C^{M_j}$ instead of the pair $(\C^{M_j},a_j)$.

\begin{example}\label{DNExample} 
Let
$$
a_j(y) = \begin{pmatrix} \mu_{j,1} & \cdots & 0 \\ \vdots & \ddots & \vdots \\ 0 & \cdots & \mu_{j,M_j} \end{pmatrix}
$$
with $\mu_{j,k} \in \R$ independent of $y \in \Omega$. In this case
$$
\varrho^{a_j(y)} =
\begin{pmatrix} \varrho^{\mu_{j,1}} & \cdots & 0 \\ \vdots & \ddots & \vdots \\ 0 & \cdots & \varrho^{\mu_{j,M_j}} \end{pmatrix},
$$
and a function
$$
p(y,\eta) = \begin{pmatrix} p_{1,1}(y,\eta) & \cdots & p_{1,M_1}(y,\eta) \\ \vdots & \ddots & \vdots \\ p_{M_2,1}(y,\eta) & \cdots & p_{M_2,M_1}(y,\eta) \end{pmatrix}
$$
belongs to $S^0_{1,\delta}(\Omega\times\R^q;(\C^{M_1},a_1),(\C^{M_2},a_2))$ if and only if the matrix entries satisfy $p_{k,l}(y,\eta) \in S_{1,\delta}^{\mu_{1,l}-\mu_{2,k}}(\Omega\times\R^q)$. This just means that $p(y,\eta)$ is a matrix that satisfies the Douglis-Nirenberg order convention \cite{DouglisNirenberg,Hormander1966}.
\end{example}

\begin{proposition}\label{basicprops}
\begin{enumerate}[(a)]
\item $S_{1,\delta}^{\mu}(\Omega\times\R^q;(\C^{M_1},a_1),(\C^{M_2},a_2))$ is a Fr\'echet space with the topology induced by the seminorms
$$
|p|_{K,\alpha,\beta} = \sup_{(y,\eta) \in K\times\R^q}\langle \eta \rangle^{-\mu+|\beta|-\delta|\alpha|}\|\langle\eta\rangle^{a_2(y)}\bigl(D_y^{\alpha}\partial_{\eta}^{\beta}p(y,\eta)\bigr)\langle \eta \rangle^{-a_1(y)}\|,
$$
where $K \Subset \Omega$ is part of a suitable countable exhaustion of $\Omega$ by compact subsets, and $\alpha,\beta \in \N_0^q$.
\item Let $a_j, a_j' \in C^{\infty}(\Omega,\End(\C^{M_j}))$, $j = 1,2$. Then there exists $\mu' > 0$ such that for every $\mu \in \R$ we have
$$
S_{1,\delta}^{\mu}(\Omega\times\R^q;(\C^{M_1},a_1),(\C^{M_2},a_2))
\subset S_{1,\delta}^{\mu + \mu'}(\Omega\times\R^q;(\C^{M_1},a'_1),(\C^{M_2},a'_2)).
$$
In particular,
$$
S^{-\infty}(\Omega\times\R^q,\Hom(\C^{M_1},\C^{M_2})) = \bigcap_{\mu \in \R} S_{1,\delta}^{\mu}(\Omega\times\R^q;(\C^{M_1},a_1),(\C^{M_2},a_2)).
$$
\item Let $p_j \in S_{1,\delta}^{\mu_j}(\Omega\times\R^q;(\C^{M_1},a_1),(\C^{M_2},a_2))$ with $\mu_j \to -\infty$ as $j \to \infty$. Let $p \in S_{1,\delta}^{\mu'}(\Omega\times\R^q,\Hom(\C^{M_1},\C^{M_2}))$ for some $\mu' \in \R$ such that $p \sim \sum_{j=1}^{\infty} p_j$. Note that such a symbol $p$ must exist by (b). Then
\begin{equation*}
p \in S_{1,\delta}^{\overline{\mu}}(\Omega\times\R^q;(\C^{M_1},a_1),(\C^{M_2},a_2)), 
\end{equation*}
where $\overline{\mu} = \max \mu_j$.
\item Differentiation $D_y^{\alpha}\partial_{\eta}^{\beta}$ of symbols induces a map
$$
S_{1,\delta}^{\mu}(\Omega\times\R^q;(\C^{M_1},a_1),(\C^{M_2},a_2)) \to
S_{1,\delta}^{\mu-|\beta|+\delta|\alpha|}(\Omega\times\R^q;(\C^{M_1},a_1),(\C^{M_2},a_2)).
$$ 
\item Pointwise composition of symbols induces a map
\begin{multline*}
S_{1,\delta}^{\mu_1}(\Omega\times\R^q;(\C^{M_2},a_2),(\C^{M_3},a_3)) \times S_{1,\delta}^{\mu_2}(\Omega\times\R^q;(\C^{M_1},a_1),(\C^{M_2},a_2)) \\ \longrightarrow
S_{1,\delta}^{\mu_1 + \mu_2}(\Omega\times\R^q;(\C^{M_1},a_1),(\C^{M_3},a_3)). 
\end{multline*}
\end{enumerate}
\end{proposition}
\begin{proof}
Assertions (a), (b), (d), and (e) follow in the usual manner. A key component for proving (a) and (b) is that for any group action $\varrho^{a(y)}$ there exists $m > 0$ such that for every compact subset $K \Subset \Omega$ we can find $C > 0$ such that $\|\langle \eta \rangle^{a(y)}\| \leq C \langle \eta \rangle^{m}$ for all $y \in K$ and all $\eta \in \R^q$. That this is indeed the case follows from the Dunford integral representation \eqref{DunfordIntegralGroup} of $\varrho^{a(y)}$.

To illustrate the argument we prove (b). Let 
\begin{equation*}
p \in S_{1,\delta}^{\mu}(\Omega\times\R^q;(\C^{M_1},a_1),(\C^{M_2},a_2)).
\end{equation*}
Then
\begin{multline*}
\|\langle \eta \rangle^{a'_2(y)}\bigl(D_y^{\alpha}\partial_{\eta}^{\beta}p(y,\eta)\bigr)\langle \eta \rangle^{-a'_1(y)}\| = \\
\|\langle \eta \rangle^{a'_2(y)}\langle \eta \rangle^{-a_2(y)}\bigl[\langle \eta \rangle^{a_2(y)}\bigl(D_y^{\alpha}\partial_{\eta}^{\beta}p(y,\eta)\bigr)\langle \eta \rangle^{-a_1(y)}\bigr]\langle \eta \rangle^{a_1(y)} \langle \eta \rangle^{-a'_1(y)}\| \leq \\
\|\langle \eta \rangle^{a'_2(y)}\|\|\langle \eta \rangle^{-a_2(y)}\|\|\langle \eta \rangle^{a_2(y)}\bigl(D_y^{\alpha}\partial_{\eta}^{\beta}p(y,\eta)\bigr)\langle \eta \rangle^{-a_1(y)}\|\|\langle \eta \rangle^{a_1(y)}\|\|\langle \eta \rangle^{-a'_1(y)}\|
\end{multline*}
Because of (2) of \eqref{Edecomp}, each of the four group terms $\langle \eta \rangle^{\pm a(y)}$ on the outside can locally uniformly in $y$ be estimated by a constant times $\langle \eta \rangle^{m}$ for all $\eta \in \R^q$ and a suitable $m > 0$. Consequently, with $\mu' = 4m$, we obtain
$$
\|\langle \eta \rangle^{a'_2(y)}\bigl(D_y^{\alpha}\partial_{\eta}^{\beta}p(y,\eta)\bigr)\langle \eta \rangle^{-a'_1(y)}\| \leq C_{K,\alpha,\beta}\langle \eta \rangle^{\mu+\mu'-|\beta|+\delta|\alpha|}
$$
with a suitable constant $C_{K,\alpha,\beta} > 0$ for all $y \in K \Subset \Omega$, and all $\eta \in \R^q$. This proves (b).

Finally, (c) is a consequence of (b).
\end{proof}

\begin{lemma}\label{KKinvsymbol}
Let $a \in C^{\infty}(\Omega;\End(\C^M))$ satisfy \eqref{Edecomp}. For every compact set $K \Subset \Omega$ and all $\alpha_j,\beta_j \in \N_0^q$, $j=1,2$, there exists a constant $C > 0$ such that
$$
\|(D^{\alpha_1}_y\partial^{\beta_1}_{\eta}\langle\eta\rangle^{a(y)})(D^{\alpha_2}_y\partial^{\beta_2}_{\eta}\langle\eta\rangle^{-a(y)})\| \leq C \langle \eta \rangle^{-|\beta_1|-|\beta_2| + \delta}
$$
for all $(y,\eta) \in K \times \R^q$. If $|\alpha_1| = |\alpha_2| = 0$ we get the estimate
$$
\|(\partial^{\beta_1}_{\eta}\langle\eta\rangle^{a(y)})(\partial^{\beta_2}_{\eta}\langle\eta\rangle^{-a(y)}) \| \leq C \langle \eta \rangle^{-|\beta_1|-|\beta_2|}
$$
for all $(y,\eta) \in K \times \R^q$.

In particular,
$$
\langle \eta \rangle^{a(y)} \in S_{1,\delta}^0(\Omega\times\R^q;(\C^M,a),\C^M) \cap S_{1,\delta}^0(\Omega\times\R^q;\C^M,(\C^M,-a)).
$$
\end{lemma}
\begin{proof}
For every $\beta \in \N_0^q$ there exist symbols $b_{\beta}, c_{\beta} \in S^{-|\beta|}(\Omega\times\R^q,\End(\C^M))$ such that
$$
\partial_{\eta}^{\beta}\langle \eta \rangle^{a(y)} = b_{\beta}(y,\eta)\langle \eta \rangle^{a(y)} \textup{ and }
\partial_{\eta}^{\beta}\langle \eta \rangle^{-a(y)} = \langle \eta \rangle^{-a(y)}c_{\beta}(y,\eta).
$$
This follows by induction, noting that
$$
\partial_{\eta_j}\langle \eta \rangle^{a(y)} = \Bigl[a(y)\frac{\partial_{\eta_j}\langle\eta\rangle}{\langle\eta\rangle}\Bigr]\langle \eta \rangle^{a(y)}
\textup{ and }
\partial_{\eta_j}\langle \eta \rangle^{-a(y)} = \langle \eta \rangle^{-a(y)}\Bigl[a(y)\frac{-\partial_{\eta_j}\langle\eta\rangle}{\langle\eta\rangle}\Bigr].
$$
In particular,
$$
(\partial^{\beta_1}_{\eta}\langle\eta\rangle^{a(y)})(\partial^{\beta_2}_{\eta}\langle\eta\rangle^{-a(y)}) = b_{\beta_1}(y,\eta)c_{\beta_2}(y,\eta),
$$
which proves the desired estimate in this case.

More generally, $(D^{\alpha_1}_y\partial^{\beta_1}_{\eta}\langle\eta\rangle^{a(y)})(D^{\alpha_2}_y\partial^{\beta_2}_{\eta}\langle\eta\rangle^{-a(y)})$ is a finite sum of terms of the form
$$
b(y,\eta)(D_y^{\gamma_1}\langle\eta\rangle^{a(y)})(D_y^{\gamma_2}\langle\eta\rangle^{-a(y)})c(y,\eta)
$$
with symbols $b \in S^{-|\beta_1|}(\Omega\times\R^q,\End(\C^M))$ and $c \in S^{-|\beta_2|}(\Omega\times\R^q,\End(\C^M))$, and $\gamma_1,\gamma_2 \in \N_0^q$.
This effectively reduces showing the claimed estimate to the case where both $|\beta_j| = 0$.

Now use the decomposition $\C^M = V_1 \oplus \dots \oplus V_N$ into the generalized eigenspaces corresponding to the eigenvalue clusters of $a(y)$ over $\Omega$, see \eqref{Edecomp}, and observe that $a$ is block-diagonal with respect to this decomposition. Let
\begin{equation*}
a_k =a|_{V_k} \in C^{\infty}(\Omega,\End(V_k))
\end{equation*}
be the part of $a$ in $V_k$, $k = 1,\dotsc,N$. For every $y \in \Omega$ the eigenvalues of $a_k(y)$ are the eigenvalues of $a(y)$ that are contained in the compact set $\Sigma_k$, and the generalized eigenspaces of $a_k(y)$ are the generalized eigenspaces of $a(y)$ corresponding to these eigenvalues. We have
$$
D_y^{\alpha}\langle \eta \rangle^{\pm a(y)} =
\begin{pmatrix}
D_y^{\alpha}\langle \eta \rangle^{\pm a_1(y)} & \cdots & 0 \\
\vdots & \ddots & \vdots \\
0 & \cdots & D_y^{\alpha}\langle \eta \rangle^{\pm a_N(y)}
\end{pmatrix},
$$
and consequently $\bigl(D_y^{\alpha_1}\langle \eta \rangle^{a(y)}\bigr)\bigl(D_y^{\alpha_2}\langle \eta \rangle^{-a(y)}\bigr)$ is given by the operator block matrix
$$
\begin{pmatrix}
\bigl(D_y^{\alpha_1}\langle \eta \rangle^{a_1(y)}\bigr)\bigl(D_y^{\alpha_2}\langle \eta \rangle^{-a_1(y)}\bigr) & \cdots & 0 \\
\vdots & \ddots & \vdots \\
0 & \cdots & \bigl(D_y^{\alpha_1}\langle \eta \rangle^{a_N(y)}\bigr)\bigl(D_y^{\alpha_2}\langle \eta \rangle^{-a_N(y)}\bigr)
\end{pmatrix}.
$$
Now use \eqref{DunfordIntegralGroup} to write
\begin{align*}
D_y^{\alpha_1}\langle\eta\rangle^{a_k(y)} &= \frac{1}{2\pi i}\int_{\Gamma_k}\langle \eta \rangle^{\lambda}D_y^{\alpha_1}(\lambda-a_k(y))^{-1}\,d\lambda, \\
D_y^{\alpha_2}\langle\eta\rangle^{-a_k(y)} &= \frac{1}{2\pi i}\int_{\Gamma_k}\langle \eta \rangle^{-\sigma}D_y^{\alpha_2}(\sigma-a_k(y))^{-1}\,d\sigma,
\end{align*}
where the contour of integration $\Gamma_k$ satisfies \eqref{SmallGammak}. Consequently
\begin{multline*}
\bigl(D_y^{\alpha_1}\langle \eta \rangle^{a_k(y)}\bigr)\bigl(D_y^{\alpha_2}\langle \eta \rangle^{-a_k(y)}\bigr)=\\
\frac{1}{(2\pi i)^2}\iint_{\Gamma_k\times\Gamma_k}\langle \eta \rangle^{\lambda-\sigma}D_y^{\alpha_1}(\lambda-a_k(y))^{-1}D_y^{\alpha_2}(\sigma-a_k(y))^{-1}\,d\lambda \, d\sigma.
\end{multline*}
The desired estimate in the case $|\beta_1| = |\beta_2| = 0$ follows from this integral representation for each of the $a_k$, $k = 1,\ldots,N$. Note that in the integral $|\lambda-\sigma| < \delta$ because $\diam(\Gamma_k)<\delta$. This finishes the proof of the lemma.
\end{proof}

\begin{remark}\label{symbolreducedtostandard}
From Proposition~\ref{basicprops} and Lemma~\ref{KKinvsymbol} we obtain that
\begin{align*}
p(y,\eta) &\in S_{1,\delta}^{\mu}(\Omega\times\R^q;(\C^{M_1},a_1),(\C^{M_2},a_2)) \\
\intertext{if and only if}
\langle \eta \rangle^{a_2(y)}p(y,\eta)\langle \eta \rangle^{-a_1(y)} &\in S_{1,\delta}^{\mu}(\Omega\times\R^q;\C^{M_1},\C^{M_2}).
\end{align*}
\end{remark}

\begin{lemma}\label{groupsymbol}
Let $b(y,\eta) \in S^{0}(\Omega\times\R^{q})$ be a scalar elliptic symbol. Assume that $b(y,\eta) > 0$ for all $(y,\eta) \in \Omega\times\R^{q}$. Then
$$
b(y,\eta)^{a(y)} \in S^{0}(\Omega\times\R^{q};\End(\C^M)).
$$
\end{lemma}
\begin{proof}
By \eqref{DunfordIntegralGroup} we have a Dunford integral representation
$$
b(y,\eta)^{a(y)} = \frac{1}{2\pi i}\int_{\Gamma}b(y,\eta)^{\sigma}(\sigma - a(y))^{-1}\,d\sigma
$$
for all $y \in \Omega$ and all $\eta \in \R^q$ with a fixed contour $\Gamma$.

Let $K \Subset \Omega$ be an arbitrary compact subset. Then there are constants $c,C > 0$ such that $c \leq b(y,\eta) \leq C$ for all $(y,\eta) \in K\times\R^{q}$, and $b^{-1}(y,\eta) \in S^{0}(\Omega\times\R^{q})$. The derivatives $\partial_y^{\alpha}\partial_{\eta}^{\beta}b(y,\eta)^{\sigma}$ are sums of products of terms $\sigma^{k}b(y,\eta)^{\sigma}$, $k \in \N_0$, and derivatives of $b^{-1}(y,\eta)$ and $b(y,\eta)$, where the sum of all orders of derivatives of $b(y,\eta)$ and $b^{-1}(y,\eta)$ with respect to $\eta \in \R^q$ that occur in each of these products is precisely $|\beta|$. Now
$$
\sup\set{|\sigma^kb(y,\eta)^{\sigma}|: \sigma \in \Gamma, \; (y,\eta) \in K\times\R^{q}} < \infty
$$
for each $k \in \N_0$. This shows that
$$
\set{b(y,\eta)^{\sigma}: \sigma \in \Gamma} \subset S^{0}(\Omega\times\R^{q})
$$
is a bounded family of symbols. Because the function $(\sigma - a(y))^{-1}$ depends smoothly on $(y,\sigma) \in \Omega\times\Gamma$, we get that
$$
\set{b(y,\eta)^{\sigma}(\sigma-a(y))^{-1}: \sigma \in \Gamma} \subset S^{0}(\Omega\times\R^{q};\End(\C^M))
$$
is bounded, which in view of the Dunford integral representation implies the lemma.
\end{proof}

\begin{definition}\label{twisteddef}
A function $p \in C^{\infty}(\Omega\times(\R^q\setminus 0),\Hom(\C^{M_1},\C^{M_2}))$ is called twisted homogeneous of degree $\mu \in \R$ with respect to the actions $\varrho^{a_j(y)}$ on $\C^{M_j}$ if
\begin{equation}\label{twistedeq}
p(y,\varrho \eta) = \varrho^{\mu} \varrho^{-a_2(y)} p(y,\eta) \varrho^{a_1(y)}
\end{equation}
for all $\varrho > 0$. A function $p \in C^{\infty}(\Omega\times\R^q,\Hom(\C^{M_1},\C^{M_2}))$ is called twisted homogeneous of degree $\mu \in \R$ in the large with respect to these actions if for every compact subset $K \Subset \Omega$ there exists $R > 0$ such that \eqref{twistedeq} holds for all $y \in K$, $|\eta| \geq R$, and all $\varrho \geq 1$. Every such function uniquely determines a twisted homogeneous function $p_{(\mu)}(y,\eta)$ on $\Omega\times(\R^q\setminus 0)$ by requiring that $p_{(\mu)}(y,\eta) = p(y,\eta)$ for $y \in \Omega$ and $|\eta|$ sufficiently large.
\end{definition}

\begin{remark}\label{homogeneousprincsymb}
Let $p(y,\eta) \in C^{\infty}(\Omega\times\R^q,\Hom(\C^{M_1},\C^{M_2}))$ be twisted homogeneous of degree $\mu \in \R$ in the large, and let $p_{(\mu)}(y,\eta)$ be twisted homogeneous of degree $\mu$ determined by $p$. Suppose there exists $\varepsilon > 0$ such that $p \in S^{\mu-\varepsilon}_{1,\delta}(\Omega\times\R^q;(\C^{M_1},a_1),(\C^{M_2},a_2))$. Then $p_{(\mu)}(y,\eta) \equiv 0$.
\end{remark}

\begin{example}
Let $[\cdot] : \R^q \to \R_+$ be $C^{\infty}$, and assume that $[\eta] = |\eta|$ for $|\eta| \geq R$ for some sufficiently large $R > 0$. If $a\in C^\infty(\Omega,\End(\C^M))$ then
$$
[\varrho\eta]^{a(y)} = [\eta]^{a(y)}\varrho^{a(y)}
$$
for all $|\eta| \geq R$ and all $\varrho \geq 1$. Consequently, the function $[\eta]^{a(y)}$ is twisted homogeneous in the large of degree zero with respect to the action generated by $a(y)$ in the domain and the trivial action $\varrho^{0} \equiv \textup{Id}$ generated by the zero endomorphism in the range. Assuming, as we are, that \eqref{Edecomp} holds for $a$ we get
$$
[\eta]^{a(y)} \in S_{1,\delta}^0(\Omega\times\R^q;(\C^M,a),\C^M)
$$
by Proposition~\ref{homogeneousissymbol} below. Writing instead
$$
[\varrho\eta]^{a(y)} = \varrho^{-(-a(y))}[\eta]^{a(y)}
$$
for $|\eta| \geq R$ and $\varrho \geq 1$ shows that we also have
$$
[\eta]^{a(y)} \in S_{1,\delta}^0(\Omega\times\R^q;\C^M,(\C^M,-a)).
$$
\end{example}

\begin{proposition}\label{homogeneousissymbol}
Let $p \in C^{\infty}(\Omega\times\R^q,\Hom(\C^{M_1},\C^{M_2}))$ be twisted homogeneous of degree $\mu \in \R$ in the large. Then $p \in S_{1,\delta}^{\mu}(\Omega\times\R^q;(\C^{M_1},a_1),(\C^{M_2},a_2))$.
\end{proposition}
\begin{proof}
Let $K \Subset \Omega$ be any compact subset. Differentiating both sides of relation \eqref{twistedeq} and multiplying by the group actions gives
\begin{gather*}
\varrho^{-\mu+|\beta|} \varrho^{a_2(y)}\bigr(D_y^{\alpha}\partial_{\eta}^{\beta}p\bigl)(y,\varrho\eta)\varrho^{-a_1(y)} = \\
\sum_{\alpha_1+\alpha_2+\alpha_3=\alpha}\frac{\alpha!}{\alpha_1 ! \alpha_2 ! \alpha_3 !}
\bigl(\varrho^{a_2(y)}\bigl[D_y^{\alpha_1}\varrho^{-a_2(y)}\bigr]\bigr)\bigl(D_y^{\alpha_2}\partial_{\eta}^{\beta}p(y,\eta)\bigr)\bigl(\bigl[D_y^{\alpha_3}\varrho^{a_1(y)}\bigr]\varrho^{-a_1(y)}\bigr).
\end{gather*}
This holds for all $y \in K$, $|\eta| \geq R$, and all $\varrho \geq 1$ for some sufficiently large $R > 0$. By Lemma~\ref{KKinvsymbol} there exists a constant $C > 0$ such that the norm of the right-hand side is bounded by $C \varrho^{\delta|\alpha|}$ as $(y,\eta)$ varies over $K \times \set{\eta \in \R^q: |\eta| = R}$ and $\varrho \geq 1$. Consequently,
$$
\Bigl\| \Bigl(\frac{|\eta|}{R}\Bigr)^{-\mu+|\beta|-\delta|\alpha|} \Bigl(\frac{|\eta|}{R}\Bigr)^{a_2(y)}\bigr(D_y^{\alpha}\partial_{\eta}^{\beta}p\bigl)(y,\eta)\Bigl(\frac{|\eta|}{R}\Bigr)^{-a_1(y)} \Bigr\|
$$
is a bounded function of $y \in K$ and $|\eta| \geq R$. Now
$$
\langle \eta \rangle^{\pm a_j(y)} = \Bigl(\frac{R\langle \eta \rangle}{|\eta|}\Bigr)^{\pm a_j(y)}\Bigl(\frac{|\eta|}{R}\Bigr)^{\pm a_j(y)} = \Bigl(\frac{|\eta|}{R}\Bigr)^{\pm a_j(y)}\Bigl(\frac{R\langle \eta \rangle}{|\eta|}\Bigr)^{\pm a_j(y)},
$$
and the function $\Bigl(\frac{R\langle \eta \rangle}{|\eta|}\Bigr)^{\pm a_j(y)}$ is bounded as $(y,\eta) \in K\times\set{\eta \in \R^q: |\eta| \geq R}$. Consequently,
$$
\Bigl\| \langle \eta \rangle^{-\mu+|\beta|-\delta|\alpha|} \langle \eta \rangle^{a_2(y)}\bigr(D_y^{\alpha}\partial_{\eta}^{\beta}p\bigl)(y,\eta)\langle \eta \rangle^{-a_1(y)} \Bigr\|
$$
is bounded for all $y \in K$ and all $|\eta| \geq R$ which implies the assertion.
\end{proof}

The following corollary is fundamental in the globalization of the pseudodifferential calculus associated with our symbol spaces. In its statement we revert to the notation in Section~\ref{sec-deltaAdmiss} and let $\phi$ and $\psi$ be $\delta$-admissible trivializations of $E\to \Y$ over open sets $\Omega$ and $\Omega'$, $a_\phi$ and $a_\psi$ as defined in Section~\ref{sec-deltaAdmiss}. 

\begin{corollary}\label{ChangeofFrame}
The element $\psi \phi^{-1}\in C^\infty(\Omega\cap \Omega',\End(\C^M))$ is twisted homogeneous of degree zero with respect to the action $\varrho^{a_\phi}$ in the domain and $\varrho^{a_\psi}$ in the range. Consequently,
\begin{equation*}
\psi\phi^{-1}\in S^0_{1,\delta}(\Omega\cap \Omega';(\C^M,a_\phi),(\C^M,a_\psi)).
\end{equation*}
\end{corollary}
This is an immediate consequence of Proposition~\ref{homogeneousissymbol} and formula \eqref{ZeroHomogeneus} which expresses the fact that $\psi\phi^{-1}$ is twisted homogeneous of degree $0$ with respect to the actions $\varrho^{a_\phi}$ and $\varrho^{a_\psi}$ on $\C^{M}$. 


\section{The operators in the local calculus}\label{sec-localoperators}

We continue our discussion under the assumptions stated in the first paragraph of Section~\ref{sec-localsymbols}.

\begin{remark}
Let $X$ be any Banach space. By $S_{1,\delta}^{\mu}(\Omega\times\R^{q},X)$ we denote as usual the space of all $p(y,\eta) \in C^{\infty}(\Omega\times\R^{q},X)$ such that for all $\alpha,\beta \in \N_0^q$ and every compact subset $K \Subset \Omega$ there exists a constant $C_{K,\alpha,\beta} > 0$ such that
$$
\|D_y^{\alpha}\partial_{\eta}^{\beta}p(y,\eta)\| \leq C_{K,\alpha,\beta}\langle \eta \rangle^{\mu - |\beta| + \delta|\alpha|}
$$
for all $(y,\eta) \in K\times{\mathbb R}^{q}$. As is customary we omit the reference to the space $X$ from the notation if $X = \C$.

By $\Psi_{1,\delta}^{\mu}(\Omega;\C^{M_1},\C^{M_2})$ we denote the space of pseudodifferential operators
$$
P : C_c^{\infty}(\Omega;\C^{M_1}) \to C^{\infty}(\Omega;\C^{M_2})
$$
given by $P = \Op(p) + R$ with
\begin{align*}
\Op(p)u(y) &= \int_{\R^q} e^{iy\eta}p(y,\eta)\hat{u}(\eta)\dbar\eta, \\
Ru(y) &= \int_{\Omega} k(y,y')u(y')\,dy'
\end{align*}
for $u \in C_c^{\infty}(\Omega;\C^{M_1})$, where $p \in S_{1,\delta}^{\mu}(\Omega\times\R^{q},\Hom(\C^{M_1},\C^{M_2}))$, and $k$ is a $C^{\infty}$-kernel taking values in $\Hom(\C^{M_1},\C^{M_2})$. The class of the symbol $p(y,\eta)$ modulo $S^{-\infty}(\Omega\times\R^{q},\Hom(\C^{M_1},\C^{M_2}))$ is uniquely determined by $P$, and we will simply refer to $p(y,\eta)$ as the symbol of $P$ with the understanding that symbols are equivalence classes modulo $S^{-\infty}$.
\end{remark}

In the following definition we take advantage of the fact that by (b) of Proposition~\ref{basicprops}, there is $\mu'$ such that 
\begin{equation}\label{InHormanderClass}
S_{1,\delta}^\mu(\Omega\times\R^q;(\C^{M_1},a_1),(\C^{M_2},a_2))\subset S_{1,\delta}^{\mu+\mu'}(\Omega\times\R^q;\C^{M_1},\C^{M_2}).
\end{equation}

\begin{definition}
Let $a_j\in C^\infty(\Omega,\End(\C^{M_j}))$ and $\mu \in \R$. We denote by 
\begin{equation*}
\Psi_{1,\delta}^{\mu}(\Omega;(\C^{M_1},a_1),(\C^{M_2},a_2))
\end{equation*}
the space of pseudodifferential operators $P : C_c^{\infty}(\Omega;\C^{M_1}) \to C^{\infty}(\Omega;\C^{M_2})$ with symbols of class $S_{1,\delta}^{\mu}(\Omega\times\R^q;(\C^{M_1},a_1),(\C^{M_2},a_2))$. The principal symbol of $P$, denoted by $\sym(P)$, is the class of the symbol $p(y,\eta)$ of $P$ modulo $S_{1,\delta}^{\mu-1+\delta}(\Omega\times\R^q;(\C^{M_1},a_1),(\C^{M_2},a_2))$.
\end{definition}

We say that $P$ has twisted homogeneous principal symbol if $\sym(P)$ has a representative that is twisted homogeneous of degree $\mu$ in the large, see Definition~\ref{twisteddef}. By Remark~\ref{homogeneousprincsymb} there is a unique function $p_{(\mu)}(y,\eta) \in C^{\infty}(\Omega\times(\R^q\setminus 0),\Hom(\C^{M_1},\C^{M_2}))$ that is twisted homogeneous of degree $\mu$ such that $p(y,\eta) = p_{(\mu)}(y,\eta)$ for every $y \in \Omega$ and all sufficiently large $|\eta|$, and $p_{(\mu)}(y,\eta)$ is independent of the choice of representative $p(y,\eta)$ of $\sym(P)$ that is twisted homogeneous in the large. In this case, we identify $\sym(P)$ with that unique twisted homogeneous function $p_{(\mu)}$ and call it the twisted homogeneous principal symbol of $P$, i.e., $\sym(P)(y,\eta) = p_{(\mu)}(y,\eta)$ is then itself considered a twisted homogeneous function of degree $\mu \in \R$ on $\Omega\times(\R^q\setminus 0)$.

\begin{proposition}\label{localcomposition}
Let $P_1 \in \Psi^{\mu_1}_{1,\delta}(\Omega;(\C^{M_2},a_2),(\C^{M_3},a_3))$ have symbol $p_1(y,\eta)$, and let $P_2 \in \Psi^{\mu_2}_{1,\delta}(\Omega;(\C^{M_1},a_1),(\C^{M_2},a_2))$ have symbol $p_2(y,\eta)$. We assume that either $P_1$ or $P_2$ is properly supported.

Then the composition $P_1\circ P_2 \in \Psi^{\mu_1+\mu_2}_{1,\delta}(\Omega;(\C^{M_1},a_1),(\C^{M_3},a_3))$ with symbol
\begin{equation}\label{SymbAsymptCompos}
p_1 \# p_2 \sim \sum_{\alpha \in \N_0^q} \frac{1}{\alpha !} \bigl(\partial_{\eta}^{\alpha}p_1\bigr)\bigl(D_y^{\alpha}p_2\bigr).
\end{equation}
In particular, the principal symbols satisfy $\sym(P_1\circ P_2) = \sym(P_1)\sym(P_2)$.
\end{proposition}

\begin{proof}
Using \eqref{InHormanderClass} we first view $P_j$ as an element of $\Psi^{\mu_j+\mu'}_{1,\delta}(\Omega;\C^{M_{3-j}},\C^{M_{4-j}})$ and conclude from the standard theory that $P_1\circ P_2\in \Psi_{1,\delta}^{\mu_1+\mu_2+2\mu'}(\Omega;\C^{M_1},\C^{M_3})$ with symbol $p_1\#p_2$ satisfying \eqref{SymbAsymptCompos}. By parts (d) and (e) of Proposition~\ref{basicprops}, 
\begin{equation*}
\bigl(\partial_{\eta}^{\alpha}p_1\bigr)\bigl(D_y^{\alpha}p_2\bigr)\in S^{\mu_1+\mu_2-(1-\delta)|\alpha|}(\Omega\times\R^q;(\C^{M_1},a_1),(\C^{M_3},a_3)),
\end{equation*}
hence by part (c) of the same proposition,
\begin{equation*}
p_1\#p_2\in S^{\mu_1+\mu_2}_{1,\delta}(\Omega\times\R^q;(\C^{M_1},a_1),(\C^{M_3},a_3))
\end{equation*}
as claimed. Consequently, $P_1\circ P_2\in \Psi^{\mu_1+\mu_2}_{1,\delta}(\Omega;\C^{M_1},a_1),(\C^{M_3},a_3)$ as claimed. The formula for the principal symbol of the composition follows immediately from \eqref{SymbAsymptCompos}.
\end{proof}

In the following proposition we shall make use of the following observation: 
Let $a\in C^\infty(\Omega;\End(\C^M))$ satisfy the conditions in \eqref{Edecomp}. Assume additionally that the direct decomposition in part (1) there is orthogonal with respect to the standard inner product of $\C^{M}$. Then the adjoint endomorphism $a^{\star} \in C^{\infty}(\Omega,\End(\C^M))$ satisfies both conditions in \eqref{Edecomp}. More precisely, the eigenvalue clusters associated with $a^{\star}$ are the complex conjugates of the ones associated with $a$, and the decomposition \eqref{Edecomp} is the same for both $a$ and $a^{\star}$.

\begin{proposition}\label{AdjointLocal}
Let $P \in \Psi^{\mu}_{1,\delta}(\Omega;(\C^{M_1},a_1),(\C^{M_2},a_2))$ have symbol $p(y,\eta)$. If the decompositions in part (1) of \eqref{Edecomp} are orthogonal, then the formal adjoint operator
\begin{equation*}
P^{\star} : C_c^{\infty}(\Omega;\C^{M_2}) \to C^{\infty}(\Omega;\C^{M_1})
\end{equation*}
defined by
$$
\int_{\Omega} \langle Pu(y),v(y) \rangle_{\C^{M_2}}\,dy = \int_{\Omega}\langle u(y),P^{\star}v(y) \rangle_{\C^{M_1}}\,dy
$$
for $u \in C_c^{\infty}(\Omega;\C^{M_1})$ and $v \in C_c^{\infty}(\Omega;\C^{M_2})$ belongs to
\begin{equation*}
\Psi^{\mu}_{1,\delta}(\Omega;(\C^{M_2},-a_2^{\star}),(\C^{M_1},-a_1^{\star}))
\end{equation*}
and has symbol
$$
q(y,\eta) \sim \sum_{\alpha\in\N_0^q}\frac{1}{\alpha!}D_y^{\alpha}\partial_{\eta}^{\alpha}p(y,\eta)^{\star}.
$$
In particular, we have $\sym(P^{\star}) = \sym(P)^{\star}$.
\end{proposition}
\begin{proof}
This again follows from Proposition~\ref{basicprops} and the standard theorem on formal adjoints in pseudodifferential calculus. Note that
$$
p(y,\eta)^{\star} \in S^{\mu}_{1,\delta}(\Omega\times\R^q;(\C^{M_2},-a_2^{\star}),(\C^{M_1},-a_1^{\star}))
$$
by Definition~\ref{symboldef} in view of the fact that $\bigl(\langle \eta \rangle^{a_j}\bigr)^{\star} = \langle \eta \rangle^{a_j^{\star}}$.
\end{proof}

\begin{definition}
A symbol $p(y,\eta) \in S^{\mu}_{1,\delta}(\Omega\times\R^q;(\C^{M_1},a_1),(\C^{M_2},a_2))$ is called \emph{elliptic} if for every compact set $K \Subset \Omega$ there exists $R > 0$ such that $p(y,\eta) : \C^{M_1} \to \C^{M_2}$ is invertible for all $y \in K$ and all $|\eta| \geq R$, and satisfies the estimate
$$
\|\langle \eta \rangle^{a_1(y)}p(y,\eta)^{-1}\langle \eta \rangle^{-a_2(y)}\| \leq C \langle \eta \rangle^{-\mu}
$$
for all $y \in K$ and all $|\eta| \geq R$ for some suitable constant $C > 0$.

An operator $P \in \Psi^{\mu}_{1,\delta}(\Omega;(\C^{M_1},a_1),(\C^{M_2},a_2))$ is elliptic if its symbol $p(y,\eta)$ is elliptic.
\end{definition}

\begin{example}
The symbol $\langle \eta \rangle^{a(y)}$ is trivially elliptic both as an element of $S_{1,\delta}^0(\Omega\times\R^q;(\C^{M},a),\C^{M})$ and $S_{1,\delta}^0(\Omega\times\R^q;\C^{M},(\C^{M},-a))$.
\end{example}

\begin{remark}\label{DNelliptdiscuss}
A symbol $p(y,\eta) \in S^{\mu}_{1,\delta}(\Omega\times\R^q;(\C^{M_1},a_1),(\C^{M_2},a_2))$ is elliptic in our sense if and only if the symbol
$$
\langle \eta \rangle^{a_2(y)}p(y,\eta)\langle \eta \rangle^{-a_1(y)} \in S_{1,\delta}^{\mu}(\Omega\times\R^q;\C^{M_1},\C^{M_2})
$$
is elliptic in the ordinary sense.

Moreover, $p(y,\eta)$ is elliptic if and only if there exists
\begin{equation*}
q(y,\eta) \in S^{-\mu}_{1,\delta}(\Omega\times\R^q;(\C^{M_2},a_2),(\C^{M_1},a_1))
\end{equation*}
such that
\begin{align*}
p(y,\eta)q(y,\eta) - 1 &\in S^{-\varepsilon}_{1,\delta}(\Omega\times\R^q;(\C^{M_2},a_2),(\C^{M_2},a_2)), \\
q(y,\eta)p(y,\eta) - 1 &\in S^{-\varepsilon}_{1,\delta}(\Omega\times\R^q;(\C^{M_1},a_1),(\C^{M_1},a_1))
\end{align*}
for some $\varepsilon > 0$. We can even arrange the remainders to be of order $-\infty$.

Consequently, our notion of ellipticity of symbols is not affected by perturbations of lower order, which implies that ellipticity for pseudodifferential operators $P \in \Psi^{\mu}_{1,\delta}(\Omega;(\C^{M_1},a_1),(\C^{M_2},a_2))$ is well defined. Moreover, it makes sense to say that $P$ is elliptic if its principal symbol $\sym(P)$ is elliptic, which means that any representative of $\sym(P)$ is elliptic.

Finally, if $p(y,\eta) \in S^{\mu}_{1,\delta}(\Omega\times\R^q;(\C^{M_1},a_1),(\C^{M_2},a_2))$ is twisted homogeneous of degree $\mu$ in the large, then $p(y,\eta)$ is elliptic if and only if the twisted homogeneous function $p_{(\mu)}(y,\eta) : \C^{M_1} \to \C^{M_2}$ determined by $p(y,\eta)$ is invertible for all $y \in \Omega$ and all $\eta \in \R^q\setminus 0$. Consequently, an operator $P \in \Psi^{\mu}_{1,\delta}(\Omega;(\C^{M_1},a_1),(\C^{M_2},a_2))$ that has a twisted homogeneous principal symbol is elliptic if and only if $\sym(P)(y,\eta)$ is invertible for all $y \in \Omega$ and all $\eta \neq 0$.
\end{remark}

\begin{example}\label{DNExample1}
Let
$$
a_j(y) = \begin{pmatrix} \mu_{j,1} & \cdots & 0 \\ \vdots & \ddots & \vdots \\ 0 & \cdots & \mu_{j,M} \end{pmatrix}
$$
with $\mu_{j,k} \in \R$ independent of $y \in \Omega$. Let
$$
p(y,\eta) = \begin{pmatrix} p_{1,1}(y,\eta) & \cdots & p_{1,M}(y,\eta) \\ \vdots & \ddots & \vdots \\ p_{M,1}(y,\eta) & \cdots & p_{M,M}(y,\eta) \end{pmatrix} \in S^0_{1,\delta}(\Omega\times\R^q;(\C^M,a_1),(\C^M,a_2)),
$$
i.e., all matrix entries satisfy $p_{k,l}(y,\eta) \in S_{1,\delta}^{\mu_{1,l}-\mu_{2,k}}(\Omega\times\R^q)$, see Example~\ref{DNExample}. Suppose that all $p_{k,l}(y,\eta)$ have homogeneous principal symbols $\sym(p_{k,l})(y,\eta)$ of degree $\mu_{1,l}-\mu_{2,k}$, and let
$$
\sym(p)(y,\eta) = \begin{pmatrix} \sym(p_{1,1})(y,\eta) & \cdots & \sym(p_{1,N})(y,\eta) \\ \vdots & \ddots & \vdots \\ \sym(p_{N,1})(y,\eta) & \cdots & \sym(p_{N,N})(y,\eta) \end{pmatrix}.
$$
Then $\sym(p)(y,\eta)$ is twisted homogeneous of degree zero with respect to the actions generated by $a_1$ and $a_2$, and an operator $P$ with symbol $p(y,\eta)$ is elliptic in our sense if and only if $\sym(p)(y,\eta)$ is invertible for all $y \in \Omega$ and all $\eta \neq 0$. This is just ellipticity in the sense of Douglis-Nirenberg \cite{DouglisNirenberg,Hormander1966} for a system represented by the symbol $p(y,\eta)$.
\end{example}

\begin{proposition}\label{parametrixlocal}
Let $P \in \Psi^{\mu}_{1,\delta}(\Omega;(\C^{M_1},a_1),(\C^{M_2},a_2))$ be elliptic. Then there exists a properly supported $Q \in \Psi^{-\mu}_{1,\delta}(\Omega;(\C^{M_2},a_2),(\C^{M_1},a_1))$ such that
$$
P\circ Q - 1 \in \Psi^{-\infty}(\Omega;\C^{M_2},\C^{M_2}) \textup{ and }
Q\circ P - 1 \in \Psi^{-\infty}(\Omega;\C^{M_1},\C^{M_1}).
$$
If $P$ has twisted homogeneous principal symbol so does $Q$, and we have
$$
\sym(Q)(y,\eta) = \sym(P)(y,\eta)^{-1} \textup{ on } \Omega\times(\R^q\setminus 0).
$$
\end{proposition}
\begin{proof}
The standard proof based on symbolic inversion modulo lower order and the formal Neumann series argument applies literally in this situation. The basic symbol properties in Proposition~\ref{basicprops}, the composition theorem Proposition~\ref{localcomposition}, and the discussion of ellipticity in Remark~\ref{DNelliptdiscuss} ensure that this is indeed the case.
\end{proof}

\begin{proposition}\label{ChangeofCoord}
Let $\chi : \Omega' \to \Omega$ be a $C^{\infty}$-diffeomorphism, and let
\begin{equation*}
P \in \Psi^{\mu}_{1,\delta}(\Omega;(\C^{M_1},a_1),(\C^{M_2},a_2)).
\end{equation*}
Then the operator pull-back
$$
\chi^*P : C_c^{\infty}(\Omega';\C^{M_1}) \to C^{\infty}(\Omega';\C^{M_2})
$$
is an operator in $\Psi^{\mu}_{1,\delta}(\Omega';(\C^{M_1},\chi^*a_1),(\C^{M_2},\chi^*a_2))$. If $P$ has symbol $p(y,\eta)$, then $\chi^*P$ has symbol $p_{\chi}(y',\eta')$ that satisfies
$$
p_{\chi}(y',\eta') \sim \sum_{\alpha\in\N_0^q}\frac{1}{\alpha !} \bigl(\partial^{\alpha}_{\eta}p\bigr)(\chi(y'),(d_{\chi(y')}\chi^{-1})^t\eta')\Phi_{\alpha}(y',\eta'),
$$
where $\Phi_{\alpha}(y',\eta')$ is a polynomial in $\eta'$ with coefficients in $C^{\infty}(\Omega')$ of degree at most $|\alpha|/2$ that depends only on the diffeomorphism $\chi$, and $\Phi_{0}(y',\eta') \equiv 1$.

The principal symbols satisfy
$$
\sym(\chi^*P)(y',\eta') = \sym(P)(\chi(y'),(d_{\chi(y')}\chi^{-1})^t\eta'),
$$
which in general needs to be interpreted as an identity of representatives modulo $S_{1,\delta}^{\mu-1+\delta}(\Omega\times\R^q;(\C^{M_1},\chi^*a_1),(\C^{M_2},\chi^*a_2))$. If $P$ has twisted homogeneous principal symbol so does $\chi^*P$, in which case this identity becomes an identity for these symbols.

The operator $\chi^*P$ is elliptic if $P$ is elliptic.
\end{proposition}
\begin{proof}
By the standard change of coordinates theorem in pseudodifferential calculus and Proposition~\ref{basicprops} we get that $\chi^*P$ is a pseudodifferential operator with symbol $p_{\chi}(y',\eta')$ that has the stated asymptotic expansion. By the form of this expansion, in order to complete the argument, it suffices to show that for every $p(y,\eta) \in S^{\mu}_{1,\delta}(\Omega\times\R^q;(\C^{M_1},a_1),(\C^{M_2},a_2))$ we have
$$
p(\chi(y'),(d_{\chi(y')}\chi^{-1})^t\eta') \in S^{\mu}_{1,\delta}(\Omega'\times\R^q;(\C^{M_1},\chi^*a_1),(\C^{M_2},\chi^*a_2)).
$$
By Remark~\ref{symbolreducedtostandard} we know that
$$
\langle \eta \rangle^{a_2(y)}p(y,\eta)\langle \eta \rangle^{-a_1(y)} \in S^{\mu}_{1,\delta}(\Omega\times\R^q;\C^{M_1},\C^{M_2}).
$$
Consequently,
\begin{equation}\label{pullbacksymb}
\begin{gathered}
\langle (d_{\chi(y')}\chi^{-1})^t\eta' \rangle^{(\chi^*a_2)(y')}p(\chi(y'),(d_{\chi(y')}\chi^{-1})^t\eta')\langle (d_{\chi(y')}\chi^{-1})^t\eta' \rangle^{-(\chi^*a_1)(y')} \\
\in S^{\mu}_{1,\delta}(\Omega'\times\R^q;\C^{M_1},\C^{M_2}).
\end{gathered}
\end{equation}
Now let
$$
b(y',\eta') = \frac{\langle \eta' \rangle}{\langle (d_{\chi(y')}\chi^{-1})^t\eta' \rangle} \in S^0(\Omega'\times\R^q).
$$
Then $b(y',\eta') > 0$ for all $(y',\eta')$, and $b(y',\eta')$ is elliptic. By Lemma~\ref{groupsymbol} we thus have
$$
b(y',\eta')^{\pm (\chi^*a_j)(y')} \in S^0(\Omega'\times\R^q;\C^{M_j},\C^{M_j}), \quad j=1,2.
$$
Multiplying \eqref{pullbacksymb} from the left and right with these terms then shows that
\begin{equation}\label{pullbacksymb1}
\langle \eta' \rangle^{(\chi^*a_2)(y')}p(\chi(y'),(d_{\chi(y')}\chi^{-1})^t\eta')\langle \eta' \rangle^{-(\chi^*a_1)(y')} \in S^{\mu}_{1,\delta}(\Omega'\times\R^q;\C^{M_1},\C^{M_2}),
\end{equation}
and consequently
$$
p(\chi(y'),(d_{\chi(y')}\chi^{-1})^t\eta') \in S^{\mu}_{1,\delta}(\Omega'\times\R^q;(\C^{M_1},\chi^*a_1),(\C^{M_2},\chi^*a_2))
$$
by Remark~\ref{symbolreducedtostandard}.

If $p(y,\eta)$ is elliptic, then $\langle \eta \rangle^{a_2(y)}p(y,\eta)\langle \eta \rangle^{-a_1(y)}$ is elliptic of order $\mu$ in the ordinary sense. Consequently, also the symbol in \eqref{pullbacksymb} is elliptic of order $\mu$ in the standard sense. Because $b(y',\eta')^{\pm (\chi^*a_j)(y')}$ is elliptic of order zero, we get that the symbol in \eqref{pullbacksymb1} is necessarily elliptic, and consequently $p(\chi(y'),(d_{\chi(y')}\chi^{-1})^t\eta')$ is elliptic.
\end{proof}

Proposition~\ref{ChangeofCoord} establishes invariance of the spaces
$$
\Psi^{\mu}_{1,\delta}(\Omega;(\C^{M_1},a_1),(\C^{M_2},a_2))
$$
under changes of coordinates. In conjunction with the following proposition we will have paved the way for the global definition of these spaces in Section~\ref{sec-global calculus}. The setting and notation in the statement are those of Section~\ref{sec-deltaAdmiss}. 

\begin{proposition}\label{InvarianceChangeOfFrame}
Let, for $j=1,2$, $\phi_j$ and $\psi_j$ be $\delta$-admissible trivializations of the vector bundles $E_j\to \Y$ over domains $\Omega$ and $\Omega'$ of local charts of $\Y$. Let $a_{j,\phi_j}=\phi_ja_j\phi_j^{-1}$, likewise $a_{j,\psi_j}=\psi_ja_j\psi_j^{-1}$ be similarly defined. View $\Omega\cap \Omega'$ as an open subset of $\R^q$ by way of either of the local charts. Define
\begin{equation*}
\Theta_j:C^\infty(\Omega\cap \Omega',\C^{M_j})\to C^\infty(\Omega\cap \Omega',\C^{M_j}),\quad \Theta_j(u)=(\psi_j\phi_j^{-1})u.
\end{equation*}
Then
\begin{equation*}
P\mapsto \Theta_2\circ P\circ\Theta_1^{-1}
\end{equation*}
is a bijection
\begin{multline*}
\Psi^{\mu}_{1,\delta}(\Omega\cap\Omega';(\C^{M_1},a_{1,\phi_1}),(\C^{M_2},a_{2,\phi_2})) \to\\ \Psi^{\mu}_{1,\delta}(\Omega\cap \Omega';(\C^{M_1},a_{1,\psi_1}),(\C^{M_2},a_{2,\psi_2})).
\end{multline*}
Further,
\begin{equation*}
\sym(\Theta_2\circ P\circ\Theta_1^{-1})= \psi_2\phi_2^{-1}\sym(P)\phi_1\psi_1^{-1}.
\end{equation*}
\end{proposition}

\begin{proof}
By Corollary~\ref{ChangeofFrame}, 
\begin{equation*}
\Theta_j\in \Psi^0_{1,\delta}(\Omega\cap \Omega';(\C^{M_j},a_{j,\phi_j}),(\C^{M_j},a_{j,\psi_j})).
\end{equation*}
Evidently $\Theta_j$ is invertible, so the conclusions follow from Proposition~\ref{localcomposition}.
\end{proof}


\section{Sobolev spaces and local regularity}\label{sec-localspaces}

We remind the reader of our standing assumption, stated in the first paragraph of Section~\ref{sec-localsymbols}.

\begin{definition}\label{SobolevSpaceLocal}
Let $s \in \R$, and let $\Lambda_{s}$ be a properly supported pseudodifferential operator on $\Omega$ with (total left) symbol $\langle \eta \rangle^{s+a(y)} = \langle \eta \rangle^s\langle \eta \rangle^{a(y)}$. Define
\begin{gather*}
H^{s+a}_{\textup{loc}}(\Omega;\C^M) = \set{u \in \mathcal D'(\Omega;\C^M): \Lambda_s u \in L^2_{\textup{loc}}(\Omega;\C^M)}, \\
H^{s+a}_{\textup{comp}}(\Omega;\C^M) = H^{s+a}_{\textup{loc}}(\Omega;\C^M) \cap {\mathcal E}'(\Omega;\C^M).
\end{gather*}
\end{definition}

Note that $\Lambda_s$ is defined without reference to a specific $\delta$, so the spaces just defined are independent of $\delta$. That in fact $\Lambda_s\in \Psi^{s}_{1,\delta}(\Omega;(\C^M,a),\C^M)$ because of our assumption on $\delta$-admissibility is of no consequence to the definition itself. The following proposition and subsequent corollary show in particular that the spaces are independent of the specific choice of operator $\Lambda_s$.

\begin{proposition}\label{LocalMappingProps}
Let $P \in \Psi^{\mu}_{1,\delta}(\Omega;(\C^{M_1},a_1),(\C^{M_2},a_2))$. Then
$$
P : H^{s+a_1}_{\textup{comp}}(\Omega;\C^{M_1}) \to H^{s-\mu+a_2}_{\textup{loc}}(\Omega;\C^{M_2}).
$$
If $P$ is properly supported, then
$$
P : \begin{cases}
H^{s+a_1}_{\textup{comp}}(\Omega;\C^{M_1}) \to H^{s-\mu+a_2}_{\textup{comp}}(\Omega;\C^{M_2}), \\
H^{s+a_1}_{\textup{loc}}(\Omega;\C^{M_1}) \to H^{s-\mu+a_2}_{\textup{loc}}(\Omega;\C^{\M_2}).
\end{cases}
$$
\end{proposition}

\begin{proof}
Let $\Lambda^{(1)}_s \in \Psi^s_{1,\delta}(\Omega;(\C^{M_1},a_1),\C^{M_1})$ be a properly supported pseudodifferential operator with symbol $\langle \eta \rangle^{s+a_1(y)}$, and let $\Lambda^{(2)}_{s-\mu} \in \Psi^{s-\mu}_{1,\delta}(\Omega;(\C^{M_2},a_2),\C^{M_2})$ be properly supported with symbol $\langle \eta \rangle^{s-\mu+a_2(y)}$. Since $\Lambda^{(1)}_s$ is elliptic there exists a properly supported $Q \in \Psi^{-s}_{1,\delta}(\Omega;\C^{M_1},(\C^{M_1},a_1))$ such that $Q\circ \Lambda^{(1)}_s = 1 + R$, where $R \in \Psi^{-\infty}(\Omega;\C^{M_1},\C^{M_1})$, see Proposition~\ref{parametrixlocal}.

Now let $P \in \Psi^{\mu}_{1,\delta}(\Omega;(\C^{M_1},a_1),(\C^{M_2},a_2))$ be properly supported, and let $u \in H^{s+a_1}_{\textup{loc}}(\Omega;\C^{M_1})$. Then
\begin{align*}
\Lambda^{(2)}_{s-\mu}(Pu) &= (\Lambda^{(2)}_{s-\mu} \circ P)((Q\circ\Lambda_s^{(1)})u - Ru) \\
&= (\Lambda^{(2)}_{s-\mu} \circ P \circ Q)(\Lambda_s^{(1)}u) - (\Lambda^{(2)}_{s-\mu} \circ P)(Ru).
\end{align*}
The operator $\Lambda^{(2)}_{s-\mu} \circ P \circ Q$ belongs to $\Psi^0_{1,\delta}(\Omega;\C^{M_1},\C^{M_2})$ by Proposition~\ref{localcomposition} and is properly supported, and consequently
$$
\Lambda^{(2)}_{s-\mu} \circ P \circ Q : L^2_{\textup{loc}}(\Omega;\C^{M_1}) \to L^2_{\textup{loc}}(\Omega;\C^{M_2}).
$$
This shows that $(\Lambda^{(2)}_{s-\mu} \circ P \circ Q)(\Lambda_s^{(1)}u) \in L^2_{\textup{loc}}(\Omega;\C^{M_2})$. On the other hand, $Ru \in C^{\infty}(\Omega;\C^{M_1})$, and thus $(\Lambda^{(2)}_{s-\mu} \circ P)(Ru) \in C^{\infty}(\Omega;\C^{M_2})$. In conclusion, we get that $\Lambda^{(2)}_{s-\mu}(Pu) \in L^2_{\textup{loc}}(\Omega;\C^{M_2})$, hence $Pu \in H^{s-\mu+a_2}_{\textup{loc}}(\Omega;\C^{M_2})$ as claimed.

The remaining mapping properties stated in the proposition follow from what we just proved by decomposing a general pseudodifferential operator in a properly supported and a smoothing part, and from the fact that properly supported pseudodifferential operators map compactly supported distributions to compactly supported distributions.
\end{proof}

\begin{corollary}\label{SobSpProps}
\begin{enumerate}[(a)]
\item $H^{s+a}_{\textup{loc}}(\Omega;\C^M) \subset H^{t+a}_{\textup{loc}}(\Omega;\C^M)$ for $ s \geq t$.
\item There exist $m,m' \geq 0$ such that
\begin{equation*}
H^{s+m'}_{\textup{loc}}(\Omega;\C^M) \subset H^{s+a}_{\textup{loc}}(\Omega;\C^M) \subset H^{s-m}_{\textup{loc}}(\Omega;\C^M)
\end{equation*}
for all $s \in \R$.
\item Let $P \in \Psi^s_{1,\delta}(\Omega;(\C^M,a),\C^M)$ be properly supported and elliptic. Then
$$
H^{s+a}_{\textup{loc}}(\Omega;\C^M) = \set{u \in \mathcal D'(\Omega;\C^M): Pu \in L^2_{\textup{loc}}(\Omega;\C^M)}.
$$
\end{enumerate}
\end{corollary}
\begin{proof}
(a) follows from Proposition~\ref{LocalMappingProps} because the identity map is an operator of class $L_{1,\delta}^{\mu}(\Omega;(\C^M,a),(\C^M,a))$ for all $\mu \geq 0$.

By Proposition~\ref{basicprops} there exists $m' > 0$ such that
$$
\Psi^s_{1,\delta}(\Omega;(\C^M,a),\C^M) \subset \Psi^{s+m'}_{1,\delta}(\Omega;\C^M,\C^M)
$$
for all $s \in \R$. Consequently, $\Lambda_s : H^{s+m'}_{\textup{loc}}(\Omega;\C^M) \to L^2_{\textup{loc}}(\Omega;\C^M)$, where $\Lambda_s$ is as in Definition~\ref{SobolevSpaceLocal}, and therefore $H^{s+m'}_{\textup{loc}}(\Omega;\C^M) \subset H^{s+a}_{\textup{loc}}(\Omega;\C^M)$. By Proposition~\ref{basicprops} there exists an $m \geq 0$ such that the identity map belongs to $\Psi^m_{1,\delta}(\Omega;(\C^M,a),\C^M)$. Consequently, $\textup{Id} : H^{s+a}_{\textup{loc}}(\Omega;\C^M) \to H^{s-m}_{\textup{loc}}(\Omega;\C^M)$ for all $s \in \R$ by Proposition~\ref{LocalMappingProps}. This proves (b).

Now let $P \in \Psi^s_{1,\delta}(\Omega;(\C^M,a),\C^M)$ be properly supported and elliptic. By Proposition~\ref{LocalMappingProps} we have
$$
H^{s+a}_{\textup{loc}}(\Omega;\C^M) \subset \set{u \in \mathcal D'(\Omega;\C^M): Pu \in L^2_{\textup{loc}}(\Omega;\C^M)}.
$$
To finish the proof of the corollary it remains to show the opposite inclusion. By Proposition~\ref{parametrixlocal} there exists a properly supported $Q \in \Psi^{-s}_{1,\delta}(\Omega;\C^M,(\C^M,a))$ such that $Q\circ P = 1 + R$ with $R \in \Psi^{-\infty}(\Omega;\C^M,\C^M)$. Let $\Lambda_s$ be as in Definition~\ref{SobolevSpaceLocal}, and let $u \in \mathcal D'(\Omega;\C^M)$ be such that $Pu \in L^2_{\textup{loc}}(\Omega:\C^M)$. Then
$$
\Lambda_s u = (\Lambda_s\circ Q)(Pu) - \Lambda_s(Ru).
$$
The operator $\Lambda_s\circ Q \in \Psi^0_{1,\delta}(\Omega;\C^M,\C^M)$ is properly supported, and consequently $(\Lambda_s\circ Q)(Pu) \in L^2_{\textup{loc}}(\Omega;\C^M)$. Clearly also $\Lambda_s(Ru) \in L^2_{\textup{loc}}(\Omega;\C^M)$ because $Ru \in C^{\infty}(\Omega;\C^M)$. This shows that $\Lambda_s u \in L^2_{\textup{loc}}(\Omega;\C^M)$, and so $u \in H^{s+a}_{\textup{loc}}(\Omega;\C^M)$.
\end{proof}

\begin{corollary}\label{LocalReg}
Let $P \in \Psi^{\mu}_{1,\delta}(\Omega;(\C^{M_1},a_1),(\C^{M_2},a_2))$ be properly supported and elliptic. Let $u \in \mathcal D'(\Omega;\C^{M_1})$ be such that $Pu = f \in H^{s+a_2}_{\textup{loc}}(\Omega;\C^{M_2})$ for some $s \in \R$. Then $u \in H^{s+\mu+a_1}_{\textup{loc}}(\Omega;\C^{M_1})$.
\end{corollary}
\begin{proof}
Let $Q \in \Psi^{-\mu}_{1,\delta}(\Omega;(\C^{M_2},a_2),(\C^{M_1},a_1))$ be a properly supported parametrix of $P$, see Proposition~\ref{parametrixlocal}. Then
$$
Q f = Q(Pu) = u + Ru \in H^{s+\mu+a_1}_{\textup{loc}}(\Omega;\C^{M_1})
$$
by Proposition~\ref{LocalMappingProps}, where $R \in \Psi^{-\infty}(\Omega;\C^{M_1},\C^{M_1})$ is properly supported. Hence $Ru \in C^{\infty}(\Omega;\C^{M_1})$, and thus $u \in H^{s+\mu+a_1}_{\textup{loc}}(\Omega;\C^{M_1})$ as asserted.
\end{proof}

\begin{example}
Let
$$
a_j(y) = \begin{pmatrix} \mu_{j,1} & \cdots & 0 \\ \vdots & \ddots & \vdots \\ 0 & \cdots & \mu_{j,M} \end{pmatrix}
$$
with $\mu_{j,k} \in \R$ independent of $y \in \Omega$, $j=1,2$. In this case,
$$
H^{s+a_j}_{\textup{loc}}(\Omega;\C^M) = \bigoplus_{k=1}^M H^{s+\mu_{j,k}}_{\textup{loc}}(\Omega).
$$
Let
$$
p(y,\eta) = \begin{pmatrix} p_{1,1}(y,\eta) & \cdots & p_{1,M}(y,\eta) \\ \vdots & \ddots & \vdots \\ p_{M,1}(y,\eta) & \cdots & p_{M,M}(y,\eta) \end{pmatrix} \in S^0_{1,\delta}(\Omega\times\R^q;(\C^M,a_1),(\C^M,a_2))
$$
be elliptic, and let $P$ be properly supported with symbol $p(y,\eta)$. By Example~\ref{DNExample1} this means that $P$ is elliptic in the sense of Douglis-Nirenberg. Corollary~\ref{LocalReg} in this case reduces to the following standard statement about regularity of solutions of $Pu = f$: If
\begin{align*}
f &= (f_1,\ldots,f_M) \in \bigoplus_{k=1}^N H^{s+\mu_{2,k}}_{\textup{loc}}(\Omega), \\
\intertext{then}
u &= (u_1,\ldots,u_M) \in \bigoplus_{k=1}^N H^{s+\mu_{1,k}}_{\textup{loc}}(\Omega).
\end{align*}
\end{example}


\section{The global calculus}\label{sec-global calculus}

Throughout this and the remaining sections of this work let $\Y$ be a smooth compact manifold without boundary of dimension $q$. We consider complex vector bundles $E \to \Y$ that are equipped with an endomorphism $a \in C^{\infty}(\Y;\End(E))$, and will typically denote the pair by $(E,a)$. The multiplicative group generated by $a$ is denoted by $\varrho^a \in C^{\infty}(\Y;\Aut(E))$, $\varrho > 0$, and $\pi$ denotes the canonical projection $T^*\Y\setminus 0\to \Y$.

\begin{definition}\label{globalops}
Let $(E_j,a_j)$, $j = 1,2$, be vector bundles over $\Y$ equipped with endomorphisms, let $0 < \delta < 1$, and let $\mu \in \R$. By $\Psi^{\mu}_{1,\delta}(\Y;(E_1,a_1),(E_2,a_2))$ we denote the space of all pseudodifferential operators
$$
P : C^{\infty}(\Y;E_1) \to C^{\infty}(\Y;E_2)
$$
of type $(1,\delta)$ with the following property:

Let $\Omega \subset \Y$ be the domain of local chart over which there are $\delta$-admissible trivializations $\phi_j : E_{j,\Omega} \to \Omega\times\C^{M_j}$ relative to $a_j$, see Section~\ref{sec-deltaAdmiss}. By way of the chart we view $\Omega$ as an open subset of $\R^q$ and require then that $P$ over $\Omega$ be represented by an operator
$$
P_{\Omega} : C_c^{\infty}(\Omega;\C^{M_1}) \to C^{\infty}(\Omega;\C^{M_2})
$$
of class $\Psi^{\mu}_{1,\delta}(\Omega;(\C^{M_1},a_{1,\phi_1}),(\C^{M_2},a_{2,\phi_2}))$.
\end{definition}

We noted in Section~\ref{sec-deltaAdmiss} that every $y_0 \in \Y$ is contained in the domain of a local chart $\Omega$ such that both bundles $E_1$ and $E_2$ have $\delta$-admissible trivializations over $\Omega$. Proposition~\ref{ChangeofCoord} (change of variables) and Proposition~\ref{InvarianceChangeOfFrame} (change of $\delta$-admissible trivializations) ensure that the class $\Psi^{\mu}_{1,\delta}(\Y;(E_1,a_1),(E_2,a_2))$ is well defined.

The class $\Psi^{\mu}_{1,\delta}(\Y;(E_1,a_1),(E_2,a_2))$ is also well defined when $\Y$ is just an open manifold, and basic properties and notions such as composition (under the usual support condition), ellipticity, existence of parametrices, and so on, are valid. However, as indicated above, we will restrict our attention here to the case where $\Y$ is closed.

Let $P \in \Psi^{\mu}_{1,\delta}(\Y;(E_1,a_1),(E_2,a_2))$. By Section~\ref{sec-deltaAdmiss} and compactness, $\Y$ has a finite covering $\Y = \bigcup_{k=1}^L \Omega_k$ by domains of local charts $\Omega_k \subset \Y$ over which there exist $\delta$-admissible trivializations of both $(E_1,a_1)$ and $(E_2,a_2)$. Let $\set{\varphi_k: k=1,\ldots,L}$ be a partition of unity subordinate to the covering of $\Y$, and choose functions $\psi_k \in C_c^{\infty}(\Omega_k)$ such that $\psi_k \equiv 1$ in a neighborhood of the support of $\varphi_k$. Write
\begin{equation}\label{PDarstellung}
P = \sum_{k=1}^L \varphi_k P \psi_k + R,
\end{equation}
where $R = \sum_{k=1}^L \varphi_k P (1-\psi_k) \in \Psi^{-\infty}(\Y;E_1,E_2)$. The operators $\varphi_k P \psi_k$ have Schwartz kernels with compact support in $\Omega_k\times\Omega_k$, and in view of Definition~\ref{globalops} their structure is described by the local calculus discussed in the previous sections. Conversely, using charts, the partition of unity, and $\delta$-admissible trivializations of the bundles, operators in $\Psi^{\mu}_{1,\delta}(\Y;(E_1,a_1),(E_2,a_2))$ can be patched out of operators in the local calculus modulo $\Psi^{-\infty}(\Y;E_1,E_2)$.

\begin{definition}
Let $P \in \Psi^{\mu}_{1,\delta}(\Y;(E_1,a_1),(E_2,a_2))$. We say that $P$ has twisted homogeneous principal symbol if every $y_0 \in \Y$ is contained in the domain $\Omega$ of a local chart such that there exist $\delta$-admissible trivializations $\phi_j : E_{j,\Omega} \to \Omega\times\C^{M_j}$ relative to $a_j$ over $\Omega$ such that the induced operator
$$
P_{\Omega} : C_c^{\infty}(\Omega;\C^{M_1}) \to C^{\infty}(\Omega;\C^{M_2})
$$
of class $\Psi^{\mu}_{1,\delta}(\Omega;(\C^{M_1},a_{1,\phi_1}),(\C^{M_2},a_{2,\phi_2}))$ has twisted homogeneous principal symbol.

The local twisted homogeneous principal symbols join and invariantly define a function $\sym(P)$ on $T^*\Y\setminus 0$ taking values in $\Hom(\pi^*E_1,\pi^*E_2)$ that satisfies the twisted homogeneity relation
\begin{equation}\label{twistedglobal}
\sym(P)(\varrho\eta_y) = \varrho^{\mu}\varrho^{-(\pi^*a_2)|_{\pi^*E_{2,y}}}\sym(P)(\eta_y)\varrho^{(\pi^*a_1)|_{\pi^*E_{1,y}}}
\end{equation}
for all $\eta_y \in T_y^*\Y\setminus 0$ and $\varrho > 0$. The global $\sym(P)$ on $T^*\Y\setminus 0$ is called the twisted homogeneous principal symbol of $P$.
\end{definition}

Let $p \in C^{\infty}(T^*\Y\setminus 0;\Hom(\pi^*E_1,\pi^*E_2))$ be twisted homogeneous of degree $\mu \in \R$, i.e., relation \eqref{twistedglobal} is satisfied. Then there exists an operator
$$
P \in \Psi^{\mu}_{1,\delta}(\Y;(E_1,a_1),(E_2,a_2))
$$
such that $\sym(P) = p$. The standard argument applies here to see this: In local coordinates and with respect to $\delta$-admissible trivializations of the bundles, we can define $P$ as the quantization of $\xi(\eta)p(y,\eta)$, where $\xi \in C^{\infty}(\R^q)$ is an excision function of the origin. The global $P$ is obtained by patching the local operators using a partition of unity.

\begin{theorem}[Composition Theorem]\label{composition}
Let $P_1 \in \Psi^{\mu_1}_{1,\delta}(\Y;(E_2,a_2),(E_3,a_3))$ and $P_2 \in \Psi^{\mu_2}_{1,\delta}(\Y;(E_1,a_1),(E_2,a_2))$. Then
$$
P_1\circ P_2 \in \Psi^{\mu_1+\mu_2}_{1,\delta}(\Y;(E_1,a_1),(E_3,a_3)).
$$
If both $P_1$ and $P_2$ have twisted homogeneous principal symbols, so does $P_1\circ P_2$ and we have $\sym(P_1\circ P_2) = \sym(P_1)\sym(P_2)$ on $T^*\Y\setminus 0$.
\end{theorem}
\begin{proof}
Let $\Omega \subset \Y$ be the domain of a local chart such that all bundles admit $\delta$-admissible trivializations $\psi_j : E_{j,\Omega} \to \Omega\times\C^{M_j}$ relative to $a_j$ over $\Omega$.
Let $\varphi_j \in C_c^{\infty}(\Omega)$, $j=1,\ldots,4$, be such that $\varphi_{j+1} \equiv 1$ in a neighborhood of the support of $\varphi_j$. Write
$$
\varphi_1(P_1\circ P_2)\varphi_2 = (\varphi_1 P_1 \varphi_3)(\varphi_4 P_2 \varphi_2) + \varphi_1 P_1 (1-\varphi_3)P_2\varphi_2.
$$
The operator $\varphi_1 P_1 (1-\varphi_3)P_2\varphi_2$ is of class $\Psi^{-\infty}(\Y;E_1,E_3)$, and both operators
$(\varphi_1 P_1 \varphi_3)$ and $(\varphi_4 P_2 \varphi_2)$ have Schwartz kernels supported in $\Omega\times\Omega$, and with respect to the trivializations $\psi_j$ of the bundles these operators are represented by operators in the classes
$$
\Psi^{\mu_1}_{1,\delta}(\Omega;(\C^{M_2},a_{2,\psi_2}),(\C^{M_3},a_{3,\psi_3})) \textup{ and }
\Psi^{\mu_2}_{1,\delta}(\Omega;(\C^{M_1},a_{1,\psi_1}),(\C^{M_2},a_{2,\psi_2})),
$$
respectively. From the local composition theorem (Proposition~\ref{localcomposition}) we obtain that $(\varphi_1 P_1 \varphi_3)(\varphi_4 P_2 \varphi_2)$ is locally represented by an operator of class
$$
\Psi^{\mu_1+\mu_2}_{1,\delta}(\Omega;(\C^{M_1},a_{1,\psi_1}),(\C^{M_3},a_{3,\psi_3})).
$$
If both $P_1$ and $P_2$ have twisted homogeneous principal symbols, so does the operator $(\varphi_1 P_1 \varphi_3)(\varphi_4 P_2 \varphi_2)$, and by our choices of the cut-offs $\varphi_j \in C_c^{\infty}(\Omega)$ we see that
$$
\sym(\varphi_1(P_1\circ P_2)\varphi_2) = \sym((\varphi_1 P_1 \varphi_3)(\varphi_4 P_2 \varphi_2)) = \varphi_1 \sym(P_1)\sym(P_2).
$$
Covering $\Y$ with suitable coordinate neighborhoods $\Omega$ and using a subordinate partition of unity proves the claim.
\end{proof}

\begin{theorem}[Formal Adjoints]\label{AdjointsGlobal}
Let $P \in \Psi^{\mu}_{1,\delta}(\Y;(E_1,a_1),(E_2,a_2))$. Fix a smooth positive density $\m$ on $\Y$, and let
$$
[ \cdot,\cdot ]_{j,y} : E_{j,y} \times F_{j,y} \to \C, \quad y \in \Y,\; j=1,2,
$$
be nondegenerate sesquilinear forms depending smoothly on $y \in \Y$. Let $a_j^{\sharp} \in C^{\infty}(\Y;\End(F_j))$ be the adjoint endomorphism of $a_j \in C^{\infty}(\Y;\End(E_j))$ with respect to $[\cdot,\cdot]_j$, $j=1,2$, i.e., $a_j^{\sharp}$ satisfies
$$
[ a_je_j,f_j ]_j = [ e_j,a_j^{\sharp}f_j ]_j
$$
for all sections $e_j \in C^{\infty}(\Y;E_j)$ and $f_j \in C^{\infty}(\Y;F_j)$. Then 
\begin{equation*}
P : C^{\infty}(\Y;E_1) \to C^{\infty}(\Y;E_2)
\end{equation*}
has a formal adjoint $P^{\sharp} : C^{\infty}(\Y;F_2) \to C^{\infty}(\Y;F_1)$ given by
$$
\int [ Pu(y),v(y)]_{2,y}\,d\m(y) = \int [ u(y),P^{\sharp}v(y) ]_{1,y}\,d\m(y)
$$
for $u \in C^{\infty}(\Y;E_1)$ and $v \in C^{\infty}(\Y;F_2)$, and $P^{\sharp} \in \Psi^{\mu}_{1,\delta}(\Y;(F_2,-a_2^{\sharp}),(F_1,-a_1^{\sharp}))$. If $P$ has twisted homogeneous principal symbol $\sym(P)$, then $P^{\sharp}$ has twisted homogeneous principal symbol $\sym(P^{\sharp}) = \sym(P)^{\sharp}$, where $\sym(P)^{\sharp} : \pi^*F_2 \to \pi^*F_1$ is the fiberwise formal adjoint of $\sym(P) : \pi^*E_1 \to \pi^*E_2$ with respect to the lifted pairings $[ \cdot,\cdot ]_j$ on $\pi^*E_j \times \pi^*F_j$, $j = 1,2$. The latter means that
$$
[\sym(P)e_1,f_2]_2 = [e_1,\sym(P)^{\sharp}f_2]_1
$$
for all sections $e_1 \in C^{\infty}(T^*\Y\setminus 0,\pi^*E_1)$ and $f_2 \in C^{\infty}(T^*\Y\setminus 0,\pi^*F_2)$.
\end{theorem}
\begin{proof}
Let $\Omega \subset \Y$ be the domain of a local chart, and let $\phi_j : E_{j,\Omega} \to \Omega\times\C^{M_j}$ be $\delta$-admissible trivializations relative to $a_j$ over $\Omega$. We equip $\C^{M_j}$ with the standard inner product $\langle \cdot,\cdot \rangle_{\C^{M_j}}$, and note that the trivializations are such that the decompositions \eqref{Edecomp} of $\C^{M_j}$ associated with $a_{j,\phi_j}$ are orthogonal with respect to $\langle \cdot,\cdot \rangle_{\C^{M_j}}$.
Now let $\phi_j^{\sharp} : \Omega\times\C^{M_j} \to F_{j,\Omega}$ be the adjoint with respect to the pairing $[\cdot,\cdot]_{j,y}$ on $E_{j,y}\times F_{j,y}$ and the standard inner product on $\C^{M_j}$, i.e.,
$$
\langle \phi_{j,y}(e), v \rangle_{\C^{M_j}} = [ e, \phi_{j,y}^{\sharp}v ]_{j,y}
$$
for $e \in E_{j,y}$ and $v \in \C^{M_j}$, where $\phi_{j,y} : E_{j,y} \to \C^{M_j}$ is the restriction of $\phi_j$ to the fiber over $y \in \Omega$, and likewise so for $\phi_j^{\sharp}$. Let $\psi_j = (\phi_j^{\sharp})^{-1} : F_{j,\Omega} \to \Omega\times\C^{M_j}$. Then $\psi_j$ is a $\delta$-admissible trivialization of $F_{j,\Omega}$ relative to $a_j^{\sharp}$ over $\Omega$, and we have $a_{j,\psi_j}^{\sharp} = (a_{j,\phi_j})^{\star} \in C^{\infty}(\Omega,\End(\C^{M_j}))$, where $\star$ represents the standard adjoint operation in $\End(\C^{M_j})$.

Now assume that $P \in \Psi^{\mu}_{1,\delta}(\Y;(E_1,a_1),(E_2,a_2))$ has Schwartz kernel that is compactly supported in $\Omega\times\Omega$. Hence $P$ is represented by an operator
$$
P_{\Omega} : C_c^{\infty}(\Omega;\C^{M_1}) \to C^{\infty}(\Omega;\C^{M_2})
$$
in the class $\Psi^{\mu}_{1,\delta}(\Omega;(\C^{M_1},a_{1,\phi_1}),(\C^{M_2},a_{2,\phi_2}))$. Proposition~\ref{AdjointLocal} is applicable here, and we get that the formal adjoint
$$
P_{\Omega}^{\star} : C_c^{\infty}(\Omega;\C^{M_2}) \to C^{\infty}(\Omega;\C^{M_1})
$$
with respect to the standard inner products on $\C^{M_j}$ is an operator in the class $\Psi^{\mu}_{1,\delta}(\Omega;(\C^{M_2},-a^{\sharp}_{2,\psi_2}),(\C^{M_1},-a^{\sharp}_{1,\psi_1}))$.
While Proposition~\ref{AdjointLocal} refers to Lebesgue measure in coordinates, a change of the density to $\m$ only results into a conjugation with a multiplication operator by a positive scalar function, and Proposition~\ref{localcomposition} then shows that this stays in the local operator class without changing the principal symbol.

The operator $P_{\Omega}^{\star}$ is the local representation of the desired operator
$$
P^{\sharp} \in \Psi^{\mu}_{1,\delta}(\Y;(F_2,-a_2^{\sharp}),(F_1,-a_1^{\sharp})).
$$
The operator $P^{\sharp}$ has compactly supported Schwartz kernel in $\Omega\times\Omega$, and we have $\sym(P^{\sharp}) = \sym(P)^{\sharp}$ in case $P$ (and then necessarily also $P^{\sharp}$) has twisted homogeneous principal symbol.

The general case reduces to considering operators with compactly supported Schwartz kernels and smoothing operators, using a partition of unity. Since any smoothing operator has a formal adjoint operator that is smoothing, the theorem is proved.
\end{proof}

\begin{definition}\label{DNEllipticityGlobal}
An operator $P \in \Psi^{\mu}_{1,\delta}(\Y;(E_1,a_1),(E_2,a_2))$ is called \emph{elliptic} if every point $y_0 \in \Y$ is contained in the domain of a local chart $\Omega$ such that there exist $\delta$-admissible trivializations $\phi_j : E_{j,\Omega} \to \Omega\times\C^{M_j}$ such that the local representation $P_{\Omega} \in \Psi^{\mu}_{1,\delta}(\Omega;(E_1,a_{1,\phi_1}),(E_2,a_{2,\phi_2}))$ of $P$ is elliptic.
\end{definition}

The notion of ellipticity in Definition~\ref{DNEllipticityGlobal} is independent of the choices of neighborhoods $\Omega$, charts, and $\delta$-admissible trivializations. This is a consequence of the local theory and its invariance properties from Section~\ref{sec-localoperators}.

If $P \in \Psi^{\mu}_{1,\delta}(\Y;(E_1,a_1),(E_2,a_2))$ has twisted homogeneous principal symbol, then $P$ is elliptic if and only if $\sym(P)$ is invertible everywhere on $T^*\Y\setminus 0$.

\begin{theorem}[Parametrix Theorem]\label{parametrixglobal}
For $P \in \Psi^{\mu}_{1,\delta}(\Y;(E_1,a_1),(E_2,a_2))$ the following are equivalent:
\begin{enumerate}[(a)]
\item $P$ is elliptic.
\item There exists $Q \in \Psi^{-\mu}_{1,\delta}(\Y;(E_2,a_2),(E_1,a_1))$ such that
$$
P\circ Q - 1 \in \Psi^{-\infty}(\Y;E_2,E_2) \textup{ and } Q\circ P - 1 \in \Psi^{-\infty}(\Y;E_1,E_1).
$$
\end{enumerate}
If $P$ has twisted homogeneous principal symbol so does the parametrix $Q$, and we have $\sym(Q) = \sym(P)^{-1}$ on $T^*\Y\setminus 0$.
\end{theorem}
\begin{proof}
If $P$ is elliptic then $\Y$ has a finite covering by open subsets such that the local representations of the restrictions of $P$ are elliptic. Proposition~\ref{parametrixlocal} applies to these representations and gives local parametrices, and we then patch a global parametrix together out of the local ones in the usual way. The converse follows from the composition theorem and the multiplicative behavior of the principal symbol in coordinates (see Proposition~\ref{localcomposition}).
\end{proof}

\begin{theorem}\label{OrderReduction}
Let $S^*\Y$ be the cosphere bundle with respect to some choice of Riemannian metric on $\Y$. Let $E$ be a vector bundle, and $a_1,a_2 \in C^{\infty}(\Y;\End(E))$ be endomorphisms. Let $r$ be the identity in $C^{\infty}(S^*\Y;\End(\pi^*E))$, extended by twisted homogeneity of degree $\mu \in \R$ with respect to the actions generated by $a_1$ and $a_2$ to all of $T^*\Y \setminus 0$. Then there exists $R \in \Psi^{\mu}_{1,\delta}(\Y;(E,a_1),(E,a_2))$ with $\sym(R) = r$ such that $R : C^{\infty}(\Y;E) \to C^{\infty}(\Y;E)$ is invertible with inverse $R^{-1} \in \Psi^{-\mu}_{1,\delta}(\Y;(E,a_2),(E,a_1))$.
\end{theorem}

Before giving the proof, we note the following. As in the standard calculus of pseudodifferential operators, our calculus allows adding a dependence on a parameter $\lambda \in \Lambda$ to the construction. For our purposes it suffices to consider $\Lambda = \R$.
In the local calculus in open sets $\Omega \subset \R^q$, the symbols of order $\mu$ in Definition~\ref{symboldef} are replaced in the parameter-dependent calculus by functions $p(y,\eta,\lambda)$ that satisfy the estimates
$$
\|\langle \eta,\lambda \rangle^{a_2(y)}\bigl(D_y^{\alpha}\partial_{(\eta,\lambda)}^{\beta}p(y,\eta,\lambda)\bigr)\langle \eta,\lambda \rangle^{-a_1(y)}\| \leq C_{K,\alpha,\beta} \langle \eta,\lambda \rangle^{\mu - |\beta| + \delta|\alpha|}
$$
for all $(y,\eta,\lambda) \in K\times\R^{q+1}$, where $K \Subset \Omega$ is any compact subset. Pseudodifferential operators with parameters in the local calculus are families $P(\lambda) = \Op(p)(\lambda) + G(\lambda) : C_c^{\infty}(\Omega;\C^{M_1}) \to C^{\infty}(\Omega;\C^{M_2})$, where $\Op(p)(\lambda)$ is the quantization of a symbol $p(y,\eta,\lambda)$ of the kind just described, and $G(\lambda)$ belongs to $\S(\Lambda,\Psi^{-\infty}(\Omega;\C^{M_1},\C^{M_2}))$, the space of Schwartz functions on $\Lambda$ with values in $\Psi^{-\infty}(\Omega;\C^{M_1},\C^{M_2})$. All constructions and results about the local calculus in Section~\ref{sec-localoperators} hold for the operator class with the added parameter, in particular Proposition~\ref{localcomposition} on composition of operator families, and Proposition~\ref{parametrixlocal} on the existence of parameter-dependent parametrices for operators that are elliptic with parameter. Ellipticity with parameter on the symbolic level means that for every compact set $K \Subset \Omega$ there exists $R > 0$ such that $p(y,\eta,\lambda)$ is invertible for all $y \in K$ and all $|(\eta,\lambda)| \geq R$, and the inverse satisfies the estimate
$$
\|\langle \eta,\lambda \rangle^{a_1(y)}p(y,\eta,\lambda)^{-1}\langle \eta,\lambda \rangle^{-a_2(y)}\| \leq C \langle \eta,\lambda \rangle^{-\mu}
$$
for all $y \in K$ and all $|\eta,\lambda| \geq R$ for some suitable constant $C > 0$. The notion of twisted homogeneity of degree $\mu \in \R$ makes sense as well and includes scaling in the parameter along with the covariables:
$$
p_{(\mu)}(y,\varrho\eta,\varrho\lambda) = \varrho^{\mu}\varrho^{-a_2(y)}p_{(\mu)}(y,\eta,\lambda)\varrho^{a_1(y)} \textup{ for } \varrho > 0 \textup{ and } (\eta,\lambda) \neq (0,0),
$$
see Definition~\ref{twisteddef}. For operator families $P(\lambda)$ with parameter-dependent twisted homogeneous principal symbol ellipticity with parameter is equivalent to the invertibility of that symbol.
The parameter-dependent calculus is also defined globally by following the same approach as in the case without parameters that is presented in this section.

\begin{proof}[Proof of Theorem~\ref{OrderReduction}]
Let $r(\lambda)$ be the identity in $\End(\pi^*E)$ on $S^*(\Y\times\Lambda)$, where $\Y\times\Lambda$ carries the product metric of the given metric on $\Y$ and the standard metric on $\Lambda = \R$, extended by twisted homogeneity of degree $\mu \in \R$ to all of $(T^*\Y\times\Lambda)\setminus 0$. Observe that the restriction $r(0)$ of $r(\lambda)$ to $\lambda = 0$ is precisely the function $r$ in the statement of the theorem. With $r(\lambda)$ we associate a family of operators $R(\lambda)$ of order $\mu \in \R$ in the parameter-dependent calculus such that $r(\lambda)$ is the parameter-dependent twisted homogeneous principal symbol of $R(\lambda)$. Then $R(\lambda)$ is elliptic with parameter $\lambda \in \R$, and consequently there exists a parameter-dependent parametrix $Q(\lambda)$ in the calculus of order $-\mu$ such that
$$
R(\lambda)\circ Q(\lambda) - 1,\; Q(\lambda)\circ R(\lambda) - 1 \in \Psi^{-\infty}(\Y,\Lambda;E).
$$
In particular, if with pick $\lambda = \lambda_0$ with $|\lambda_0|$ sufficiently large, then $R(\lambda_0)$ is invertible with inverse $R(\lambda_0)^{-1} = Q(\lambda_0) + G$ for some appropriate $G \in \Psi^{-\infty}(\Y;E)$. The pseudodifferential operator $R(\lambda_0)$ is an element of order $\mu \in \R$ in the calculus without parameters, and its inverse $R(\lambda_0)^{-1}$ is an element of order $-\mu$. By construction of the operator $R(\lambda_0)$ we see that it does have a twisted homogeneous principal symbol on $T^*\Y\setminus 0$ that is simply given by $r$. Hence the assertion of the theorem holds with $R = R(\lambda_0)$.
\end{proof}


\section{Sobolev spaces and Fredholm theory}\label{sec-Fredholm theory}

We continue our investigation with the definition of the global Sobolev spaces on $\Y$, the mapping properties of the operators in the calculus in the Sobolev space scale, and the Fredholm theory of elliptic operators.

\begin{definition}
Let $(E,a)$ be a vector bundle equipped with an endomorphism $a$. For $s \in \R$ let
$$
H^{s+a}(\Y;E)
$$
be the space of all $u \in \mathcal D'(\Y;E)$ such that over domains $\Omega \subset \Y$ of local charts over which there exists a $\delta$-admissible trivialization $\phi : E_{\Omega} \to \Omega\times\C^M$ relative to $a$, the restriction $u|_{\Omega}$ is a distribution in $H_{\textup{loc}}^{s+a_{\phi}}(\Omega;\C^M)$.
\end{definition}

By the comment following Definition~\ref{SobolevSpaceLocal} and by Proposition~\ref{LocalMappingProps}, Corollary~\ref{SobSpProps}, and the invariance properties of the local calculus we see that the space $H^{s+a}(\Y;E)$ is well defined and is independent of the choice of $0 < \delta < 1$.

\begin{theorem}\label{GlobalMappingProps}
\begin{enumerate}[(a)]
\item Let $P \in \Psi^{\mu}_{1,\delta}(\Y;(E_1,a_1),(E_2,a_2))$. Then
\begin{equation}\label{PGlobalMap}
P : H^{s+a_1}(\Y;E_1) \to H^{s-\mu+a_2}(\Y;E_2)
\end{equation}
for every $s \in \R$.
\item Fix a smooth positive density on $\Y$ and a Hermitian metric on $E$. Let $\Lambda_s \in L_{1,\delta}^s(\Y;(E,a),E)$ be invertible with inverse $\Lambda_s^{-1} \in \Psi^{-s}_{1,\delta}(\Y;E,(E,a))$, see Theorem~\ref{OrderReduction}. Then $H^{s+a}(\Y;E)$ is a Hilbert space with respect to the inner product
$$
\langle u,v \rangle = \langle \Lambda_s u,\Lambda_s v \rangle_{L^2(\Y;E)}.
$$
The topology induced on $H^{s+a}(\Y;E)$ by the norm associated to this inner product is independent of the choice of density on $\Y$ and Hermitian form on $E$, and independent of the choice of $\Lambda_s$. The map \eqref{PGlobalMap} is continuous with respect to this topology. 
\item $C^{\infty}(\Y;E) \hookrightarrow H^{s+a}(\Y;E) \hookrightarrow \mathcal D'(\Y;E)$ continuously, and $C^{\infty}(\Y;E)$ is dense in $H^{s+a}(\Y;E)$ for every $s \in \R$.
\item $H^{s+a}(\Y;E) \hookrightarrow H^{t+a}(\Y;E)$ continuously for $s \geq t$, and this embedding is compact for $s > t$.
\end{enumerate}
\end{theorem}
\begin{proof}
Write $P = \sum_{k=1}^L\varphi_kP\psi_k + R$ as in \eqref{PDarstellung}. The operators $\varphi_kP\psi_k$ are pull-backs of operators in the local calculus with compactly supported Schwartz kernels, and consequently Proposition~\ref{LocalMappingProps} implies that $\varphi_k P \psi_k : H^{s+a_1}(\Y;E_1) \to H^{s-\mu+a_2}(\Y;E_2)$ for each $k = 1,\ldots,L$. On the other hand, $R$ is smoothing and thus trivially has the desired mapping properties. This proves (a).

(b) follows from the continuity of all pseudodifferential operators acting in distributions, the Composition Theorem~\ref{composition} for the calculus, and the boundedness of pseudodifferential operators of order zero and type $(1,\delta)$ in $L^2$. (c) is evident, and by utilizing Theorem~\ref{OrderReduction} part (d) reduces to the familiar result that pseudodifferential operators of order $< 0$ and type $(1,\delta)$ are compact in $L^2$.
\end{proof}

\begin{theorem}
Fix a smooth positive density $\m$ on $\Y$, and let
$$
[\cdot,\cdot]_y : E_y \times F_y \to \C, \quad y \in \Y,
$$
be a nondegenerate sesquilinear form depending smoothly on $y \in \Y$. The map
$$
\{ u,v \} = \int [ u(y),v(y) ]_{y}\,d\m(y)
$$
for $u \in C^{\infty}(\Y;E)$ and $v \in C^{\infty}(\Y;F)$ extends by continuity to a nondegenerate sesquilinear form
$$
\{ \cdot,\cdot \} : H^{s+a}(\Y;E) \to H^{-s-a^{\sharp}}(\Y;F) \to \C
$$
that induces an antilinear isomorphism $H^{s+a}(\Y;E)' \cong H^{-s-a^{\sharp}}(\Y;F)$. Here $a^{\sharp} \in C^{\infty}(\Y;\End(F))$ is the adjoint endomorphism of $a \in C^{\infty}(\Y;\End(E))$ with respect to $[\cdot,\cdot]_y$.
\end{theorem}
\begin{proof}
By Theorem~\ref{OrderReduction} there exists an invertible $\Lambda_s \in \Psi^s_{1,\delta}(\Y;(E,a),E)$ with inverse $\Lambda_s^{-1} \in \Psi^{-s}_{1,\delta}(\Y;E,(E,a))$. By Theorem~\ref{AdjointsGlobal} we have $$
\bigl(\Lambda_s^{-1}\bigr)^{\sharp} \in \Psi^{-s}_{1,\delta}(\Y;(F,-a^{\sharp}),F).
$$
For $u \in C^{\infty}(\Y;E)$ and $v \in C^{\infty}(\Y;F)$ we have
$$
\{ u,v \} = \int [ \Lambda_s^{-1}\Lambda_su(y),v(y) ]_{y}\,d\m(y) = \int [ \Lambda_su(y),\bigl(\Lambda_s^{-1}\bigr)^{\sharp}v(y) ]_{y}\,d\m(y).
$$
Theorem~\ref{GlobalMappingProps} shows that the right-hand side extends by continuity to all $u \in H^{s+a}(\Y;E)$ and $v \in H^{-s-a^{\sharp}}(\Y;F)$, and the extension $\{\cdot,\cdot\}$ has the desired properties.
\end{proof}

\begin{theorem}[Elliptic Regularity]
Let $P \in \Psi^{\mu}_{1,\delta}(\Y;(E_1,a_1),(E_2,a_2))$ be elliptic. Let $u \in \mathcal D'(\Y;E_1)$ be such that $Pu = f \in H^{s+a_2}(\Y;E_2)$ for some $s \in \R$. Then $u \in H^{s+\mu+a_1}(\Y;E_1)$.
\end{theorem}
\begin{proof}
This follows from the Parametrix Theorem~\ref{parametrixglobal} and Theorem~\ref{GlobalMappingProps} in the usual way. The argument is the same as in Corollary~\ref{LocalReg} for the local calculus.
\end{proof}

\begin{theorem}[Fredholm Theorem]\label{FredholmTheorem}
Let $P \in \Psi^{\mu}_{1,\delta}(\Y;(E_1,a_1),(E_2,a_2))$. The following are equivalent:
\begin{enumerate}[(a)]
\item $P$ is elliptic.
\item $P : H^{s+a_1}(\Y;E_1) \to H^{s-\mu+a_2}(\Y;E_2)$ is a Fredholm operator for every $s \in \R$.
\item $P : H^{s_0+a_1}(\Y;E_1) \to H^{s_0-\mu+a_2}(\Y;E_2)$ is a Fredholm operator for some $s_0 \in \R$.
\end{enumerate}
\end{theorem}
\begin{proof}
Let $R_1 \in \Psi^0_{1,\delta}(\Y;E_1,(E_1,a_1))$ and $R_2 \in \Psi^0_{1,\delta}(\Y;(E_2,a_2),E_2)$ be elliptic and invertible, and suppose that the inverses satisfy $R_1^{-1} \in \Psi^0_{1,\delta}(\Y;(E_1,a_1),E_1)$ and $R_2^{-1} \in \Psi^0_{1,\delta}(\Y;E_2,(E_2,a_2))$, respectively. Such operators exist according to Theorem~\ref{OrderReduction}. Then each of the stated properties for $P$ is equivalent to the corresponding property for the operator $R_2PR_1 \in \Psi^{\mu}_{1,\delta}(\Y;E_1,E_2)$. Consequently, the proof of Theorem~\ref{FredholmTheorem} reduces to the standard result where both $a_1$ and $a_2$ are the zero endomorphisms, and $P$ is an operator of order $\mu$ and type $(1,\delta)$.
\end{proof}

\begin{corollary}[Spectral Invariance]\label{SpectralInvariance}
Let $P \in \Psi^{\mu}_{1,\delta}(\Y;(E_1,a_1),(E_2,a_2))$, and suppose that
$$
P : H^{s_0+a_1}(\Y;E_1) \to H^{s_0-\mu+a_2}(\Y;E_2)
$$
is invertible for some $s_0 \in \R$. Then $P^{-1} \in \Psi^{-\mu}_{1,\delta}(\Y;(E_2,a_2),(E_1,a_1))$.
\end{corollary}
\begin{proof}
By Theorem~\ref{FredholmTheorem}, $P$ is elliptic. Let $Q \in \Psi^{-\mu}_{1,\delta}(\Y;(E_2,a_2),(E_1,a_1))$ be a parametrix such that $P\circ Q = 1 + R_r$ and $Q\circ P = 1 + R_l$, where $R_l$ and $R_r$ are smoothing, see Theorem~\ref{parametrixglobal}. Then
$$
P^{-1} = Q - Q\circ R_r + R_l\circ P^{-1}\circ R_r,
$$
and $R_l\circ P^{-1}\circ R_r$ is smoothing because it extends to an operator that maps distributions to $C^{\infty}$-functions. Consequently, $P^{-1} \in \Psi^{-\mu}_{1,\delta}(\Y;(E_2,a_2),(E_1,a_1))$ as desired.
\end{proof}

\begin{corollary}[Functional Calculus]\label{FunctionalCalculus}
Let $P \in \Psi^{0}_{1,\delta}(\Y;(E,a),(E,a))$. Then the spectrum $\Sigma$ of the bounded operator $P : H^{s+a}(\Y;E) \to H^{s+a}(\Y;E)$ is independent of $s \in \R$.

If $f$ is a holomorphic function in a neighborhood of $\Sigma$, then the operator $f(P)$ defined via the holomorphic functional calculus belongs to $\Psi^{0}_{1,\delta}(\Y;(E,a),(E,a))$.
\end{corollary}
\begin{proof}
Independence of the spectrum of $s \in \R$ follows at once from Corollary~\ref{SpectralInvariance}. Moreover, $\Psi^{0}_{1,\delta}(\Y;(E,a),(E,a))$ carries a natural Fr\'echet topology such that
$$
\Psi^{0}_{1,\delta}(\Y;(E,a),(E,a)) \hookrightarrow \L(H^{s+a}(\Y;E)),
$$
and whenever $P \in \Psi^{0}_{1,\delta}(\Y;(E,a),(E,a))$ is invertible in $\L(H^{s+a}(\Y;E))$ the inverse belongs to $\Psi^{0}_{1,\delta}(\Y;(E,a),(E,a))$. Consequently, $\Psi^{0}_{1,\delta}(\Y;(E,a),(E,a))$ is a $\Psi$-algebra in $\L(H^{s+a}(\Y;E))$ in the sense of \cite{Gramsch1984}, and therefore invariant with respect to holomorphic functional calculus.
\end{proof}

\begin{theorem}[Index Theorem]
Let $P \in \Psi^{\mu}_{1,\delta}(\Y;(E_1,a_1),(E_2,a_2))$ be elliptic, and suppose that $P$ has twisted homogeneous principal symbol $\sym(P)$ on $T^*\Y\setminus 0$. Then
\begin{equation*}
0 \to \pi^*E_1 \xrightarrow{\sym(P)} \pi^*E_2 \to 0
\end{equation*}
is a short exact sequence outside the zero section on $T^*\Y$ and consequently induces an element $[\sym(P)]$ in the $K$-group $K(T^*\Y)$ with compact support. The Fredholm index $\ind(P)$ of the operator $P : H^{s+a_1}(\Y;E_1) \to H^{s-\mu+a_2}(\Y;E_2)$ is given by
$$
\ind(P) = \tind([\sym(P)]),
$$
where $\tind : K(T^*\Y) \to {\mathbb Z}$ is the topological index map, see \cite{AtiyahSinger1,AtiyahSinger3}.
\end{theorem}
\begin{proof}
Let $S^*\Y$ be the cosphere bundle with respect to some Riemannian metric, and let $h$ be the restriction of $\sym(P)$ to $S^*\Y$. Then $\sym(P)$ is obtained from $h$ via extension by twisted homogeneity of degree $\mu$ with respect to the pull-backs of the actions $\varrho^{a_j}$ on $E_j$, $j = 1,2$, see \eqref{twistedglobal}. For $0 \leq t \leq 1$ define $H(t,\cdot)$ by extending $h$ by twisted homogeneity of degree $\mu$ with respect to the actions $\varrho^{ta_j}$ on $E_j$, $j=1,2$, to all of $T^*\Y\setminus 0$. Then $H(1,\cdot) = \sym(P)$, and $q = H(0,\cdot)$ is an ordinary homogeneous bundle isomorphism of degree $\mu$. By construction,
\begin{equation*}
0 \to \pi^*E_1 \xrightarrow{H}\pi^*E_2 \to 0
\end{equation*}
is exact on $[0,1]\times T^*\Y$ away from $[0,1]\times 0$, and consequently $[q] = [\sym(P)] \in K(T^*\Y)$.
Now pick
$$
R_1 \in \Psi^0_{1,\delta}(\Y;E_1,(E_1,a_1)) \textup{ with } R_1^{-1} \in \Psi^0_{1,\delta}(\Y;(E_1,a_1),E_1)
$$
such that $\sym(R_1)|_{S^*\Y} = \textup{Id}_{\pi^*E_1}$, and likewise
$$
R_2 \in \Psi^0_{1,\delta}(\Y;(E_2,a_2),E_2) \textup{ with } R_2^{-1} \in \Psi^0_{1,\delta}(\Y;E_2,(E_2,a_2))
$$
with $\sym(R_2)|_{S^*\Y} = \textup{Id}_{\pi^*E_2}$; the existence of such operators is guaranteed by Theorem~\ref{OrderReduction}. Then
$$
Q = R_2\circ P \circ R_1 \in \Psi^{\mu}_{1,\delta}(\Y;E_1,E_2),
$$
and
$$
\sym(R_2\circ P \circ R_1) = \sym(R_2)\sym(P)\sym(R_1) = q.
$$
The latter relation for the principal symbols is true because the restriction of $\sym(R_2)\sym(P)\sym(R_1)$ to $S^*\Y$ equals $h$, and $\sym(R_2)\sym(P)\sym(R_1)$ is homogeneous of degree $\mu$ (without twisting). The Atiyah-Singer Index Theorem \cite{AtiyahSinger1} now implies that the Fredholm index of the operator
$Q : H^s(\Y;E_1) \to H^{s-\mu}(\Y;E_2)$ is given by
$$
\ind(Q) = \tind([q]) = \tind([\sym(P)]).
$$
On the other hand, since both
$$
R_1 : H^s(\Y;E_1) \to H^{s+a_1}(\Y;E_1) \textup{ and }
R_2 : H^{s-\mu+a_2}(\Y;E_2) \to H^{s-\mu}(\Y;E_2)
$$
are isomorphisms, we see that
$$
\ind\bigl(Q : H^s(\Y;E_1) \to H^{s-\mu}(\Y;E_2)\bigr) =
\ind\bigl(P : H^{s+a_1}(\Y;E_1) \to H^{s-\mu+a_2}(\Y;E_2)\bigr).
$$
This finishes the proof of the theorem.
\end{proof}


\section{Toeplitz operators}\label{sec-Toeplitz Operators}

The following lemma utilizes standard arguments from $K$-theory of operator algebras. The results on spectral invariance and holomorphic functional calculus from the previous section insure that they are applicable here.

\begin{lemma}\label{projectionsincalculus}
Let $\wp : \pi^*E \to \pi^*E$ be a projection on $T^*\Y\setminus 0$ that is twisted homogeneous of degree zero with respect to the action generated by $a \in C^{\infty}(\Y;\End(E))$. Then there is a projection $\Pi = \Pi^2 \in \Psi^0_{1,\delta}(\Y;(E,a),(E,a))$ such that $\sym(\Pi)=\wp$.
\end{lemma}
\begin{proof}
Let ${\mathfrak P} \in \Psi^0_{1,\delta}(\Y;(E,a),(E,a))$ with $\sym({\mathfrak P}) = \wp$. Then ${\mathfrak P}^2 - {\mathfrak P}$ is an operator in $L_{1,\delta}^{-1+\delta}(\Y;(E,a),(E,a))$, and consequently
$$
{\mathfrak P}^2 - {\mathfrak P} : H^{a}(\Y;E) \to H^{a}(\Y;E)
$$
is compact. By analytic Fredholm theory, the spectrum of ${\mathfrak P} \in \L(H^{a}(\Y;E))$ is discrete in $\C \setminus \{0,1\}$, and consequently there exists $0 < \varepsilon < 1$ such that
$\spec({\mathfrak P})\cap\partial B_{\varepsilon}(1) = \emptyset$. Define
$$
\Pi = \frac{1}{2\pi i }\int\limits_{\partial B_{\varepsilon}(1)} (\sigma - {\mathfrak P})^{-1}\,d\sigma \in \L(H^{a}(\Y;E)).
$$
Then $\Pi = \Pi^2$, and by Corollary~\ref{FunctionalCalculus} we have $\Pi \in \Psi^0_{1,\delta}(\Y;(E,a),(E,a))$, and $\sym(\Pi) = \wp$.
\end{proof}

Lemma~\ref{projectionsincalculus} guarantees that the projections $\Pi_k$ with prescribed twisted homogeneous principal symbols alluded to in the assumptions of the following theorem exist in the calculus.

\begin{theorem}
Fix a Riemannian metric on $\Y$, and let
$$
P \in \Psi^{\mu}_{1,\delta}(\Y;(E_1,a_1),(E_2,a_2))
$$
have twisted homogeneous principal symbol $\sym(P) : \pi^*E_1 \to \pi^*E_2$. Suppose that there are subbundles $J_1 \subset \pi^*E_1\big|_{S^*\Y}$ and $J_2 \subset \pi^*E_2\big|_{S^*\Y}$ such that $\sym(P) : J_1 \to J_2$ is invertible over $S^*\Y$. Let $\wp_k \in C^{\infty}(S^*\Y;\pi^*E_k\big|_{S^*\Y})$ be bundle projections $\pi^*E_k\big|_{S^*\Y} \to J_k$, $k=1,2$, and let $\Pi_k=\Pi_k^2 \in \Psi^0_{1,\delta}(\Y;(E_k,a_k),(E_k,a_k))$ with $\sym(\Pi_k) = \wp_k$ on $S^*\Y$. Then there exists $Q \in \Psi^{-\mu}_{1,\delta}(\Y;(E_2,a_2),(E_1,a_1))$ having twisted homogeneous principal symbol such that
\begin{align*}
\big(\Pi_2 P \Pi_1\bigr)\circ\bigl(\Pi_1 Q \Pi_2\bigr) &= \Pi_2 + \bigl(\Pi_2R_2\Pi_2\bigr), \\
\bigl(\Pi_1 Q \Pi_2\bigr)\circ\big(\Pi_2 P \Pi_1\bigr) &= \Pi_1 + \bigl(\Pi_1R_1\Pi_1\bigr)
\end{align*}
with $R_k \in \Psi^{-\infty}(\Y;E_k,E_k)$, $k = 1,2$. In particular,
$$
\Pi_2 P \Pi_1 : \Pi_1H^{s+a_1}(\Y;E_1) \to \Pi_2H^{s-\mu+a_2}(\Y;E_2)
$$
is Fredholm for every $s \in \R$, and $\Pi_1Q\Pi_2$ is a Fredholm inverse.
\end{theorem}
\begin{proof}
Let
$$
{\mathscr P} = \begin{pmatrix} {\mathscr P}_{1,1} & {\mathscr P}_{1,2} \\ {\mathscr P}_{2,1} & {\mathscr P}_{2,2} \end{pmatrix} \in \Psi^{\mu}_{1,\delta}\Biggl(\Y;\Biggl(\begin{array}{c} E_1 \\ \oplus \\ E_2 \end{array},\begin{pmatrix} a_1 & 0 \\ 0 & a_2 \end{pmatrix}\Biggr),\Biggl(\begin{array}{c} E_1 \\ \oplus \\ E_2 \end{array},\begin{pmatrix} a_1 & 0 \\ 0 & a_2 \end{pmatrix}\Biggr)\Biggr)
$$
with ${\mathscr P}_{i,j} : C^{\infty}(\Y;E_j) \to C^{\infty}(\Y;E_i)$ have twisted homogeneous principal symbol $\sym({\mathscr P})$ such that the restriction of $\sym({\mathscr P})$ to $S^*\Y$ is given by
$$
\sym({\mathscr P}) = \begin{pmatrix}
1 - \wp_1 & \wp_1[\sym(P) : J_1\to J_2]^{-1}\wp_2 \\
\wp_2\sym(P)\wp_1 & 1 - \wp_2
\end{pmatrix} :
\begin{array}{c} \pi^*E_1 \\ \oplus \\ \pi^*E_2 \end{array}
\to
\begin{array}{c} \pi^*E_1 \\ \oplus \\ \pi^*E_2 \end{array}.
$$
We further pick the lower left corner of ${\mathscr P}$ to be ${\mathscr P}_{2,1} = \Pi_2P\Pi_1$, and ${\mathscr P}_{k,k} = (1 - \Pi_k){\mathscr P}_{k,k}(1-\Pi_k)$ for $k = 1,2$.

With this definition, our assumption on $\sym(P)$ implies that ${\mathscr P}$ is elliptic, and by Theorem~\ref{parametrixglobal} there exists a parametrix
$$
{\mathscr Q} = \begin{pmatrix} {\mathscr Q}_{1,1} & {\mathscr Q}_{1,2} \\ {\mathscr Q}_{2,1} & {\mathscr Q}_{2,2} \end{pmatrix} \in \Psi^{-\mu}_{1,\delta}\Biggl(\Y;\Biggl(\begin{array}{c} E_1 \\ \oplus \\ E_2 \end{array},\begin{pmatrix} a_1 & 0 \\ 0 & a_2 \end{pmatrix}\Biggr),\Biggl(\begin{array}{c} E_1 \\ \oplus \\ E_2 \end{array},\begin{pmatrix} a_1 & 0 \\ 0 & a_2 \end{pmatrix}\Biggr)\Biggr)
$$
of ${\mathscr P}$ modulo smoothing remainders. The operator $Q = {\mathscr Q}_{1,2}$ has the asserted properties.
\end{proof}


\end{document}